\documentclass[11pt,twoside,letter]{article}
\usepackage{amsthm, amsmath, amssymb, amsfonts, graphicx, epsfig}
\usepackage{accents}
\usepackage{bbm}
\usepackage{mathrsfs}
\usepackage{epstopdf}
\usepackage{graphicx}
\usepackage[font=sl,labelfont=bf]{caption}
\usepackage{subcaption}
\usepackage{float}
\restylefloat{table}
\usepackage{enumerate}
\usepackage[usenames,dvipsnames]{color}
\usepackage[colorlinks=true, pdfstartview=FitV, linkcolor=blue,
            citecolor=blue, urlcolor=blue]{hyperref}
\usepackage[usenames]{color}
\usepackage{url}
\makeatletter\def\url@leostyle{%
 \@ifundefined{selectfont}{\def\UrlFont{\sf}}{\def\UrlFont{\scriptsize\ttfamily}}} \makeatother
\usepackage[framemethod=default]{mdframed}
\usepackage{array}
\usepackage[margin=1.0in, letterpaper]{geometry}
\newtheorem{theorem}{Theorem}
\newtheorem{proposition}[theorem]{Proposition}
\newtheorem{lemma}[theorem]{Lemma}

\theoremstyle{definition}
\newtheorem{definition}[theorem]{Definition}

\theoremstyle{remark}
\newtheorem{remark}[theorem]{Remark}

\numberwithin{equation}{section}
\numberwithin{theorem}{section}
\definecolor{Red}{rgb}{0.9,0,0.0}
\definecolor{Blue}{rgb}{0,0.0,1.0}
\def\cB{\mathcal{B}}

\def\cD{\mathcal{D}}

\def\cN{\mathcal{N}}

\def\cQ{\mathcal{Q}}

\def\bE{\mathbb{E}}

\def\bN{\mathbb{N}}

\def\bP{\mathbb{P}}

\def\bR{\mathbb{R}}

\def\sE{\mathscr{E}}
\def\sF{\mathscr{F}}

\newcommand{\wt}{\widetilde}
\newcommand{\wh}{\widehat}

\newcommand{\1}{\mathbbm{1}}            % preferable way of writing indicator function
\renewcommand{\mid}{\;|\;}              % mid bar with small spaces before and after: x | y
\DeclareMathOperator{\dif}{d \!}        % used for differential, same as in commath.sty
\DeclareMathOperator*{\argmin}{arg\,min} % argmin
\DeclareMathOperator{\Var}{Var}          % \Var for variance
\def\ed{\;{\stackrel{\cD}{=}}\;}
\def\cd{\;{\stackrel{\cD}{\longrightarrow}}\;}

\title{Bayesian Estimations for Diagonalizable Bilinear SPDEs}

\author{
    Ziteng Cheng \\[-0.3ex]
    \url{zcheng7@hawk.iit.edu} \\[-0.9ex]
 \and
    Igor Cialenco \\[-0.3ex]
    \url{cialenco@iit.edu}  \\[-0.9ex]
    \url{http://math.iit.edu/\~igor}
 \and
     Ruoting Gong \\[-0.3ex]
    \url{rgong2@iit.edu}  \\[-0.9ex]
    \url{http://mypages.iit.edu/\~rgong2}
 \and \\
        {\footnotesize Department of Applied Mathematics, Illinois Institute of Technology} \\
        {\footnotesize W 32nd Str, John T. Rettaliata Engineering Center, Room 208, Chicago, IL 60616, USA}\\
        }
\date{ {\small  %\today\\ Preliminary Draft}} %
First Circulated: May 29, 2018\\ 
This Version: March 1, 2019 
}}

\begin{document}

\maketitle

\vspace{-2em}

\smallskip

{\footnotesize
\begin{tabular}{l@{} p{350pt}}
  \hline \\[-.2em]
  \textsc{Abstract}: \ &
The main goal of this paper is to study the parameter estimation problem, using the Bayesian methodology, for the drift coefficient of some linear (parabolic) SPDEs driven by a multiplicative noise of special structure. We take the spectral approach by assuming that one path of the first $N$ Fourier modes of the solution is continuously observed over a finite time interval. First, we show that the model is regular and fits into classical local asymptotic normality framework, and thus the MLE and the Bayesian estimators are weakly consistent, asymptotically normal, efficient, and asymptotically equivalent in the class of loss functions with polynomial growth. Secondly, and mainly, we prove a Bernstein-Von Mises type result, that strengthens the existing results in the literature, and that also allows to investigate the Bayesian type estimators with respect to a larger class of priors and loss functions than that covered by classical asymptotic theory. In particular, we prove strong consistency and asymptotic normality of Bayesian estimators in the class of loss functions of at most exponential growth. Finally, we present some numerical examples that illustrate the obtained theoretical results.

 \\[0.5em]
\textsc{Keywords:} \ &  statistical inference for SPDEs, Bayesian statistics, Bernstein-Von Mises, parabolic SPDE, multiplicative noise, stochastic evolution equations, identification problems for SPDEs \\
\textsc{MSC2010:} \ &  60H15, 65L09, 62M99\\[1em]
  \hline
\end{tabular}
}

\section{Introduction}
The analytical theory for Stochastic Partial Differential Equations (SPDEs) has been extensively studied over the past few decades. It is well recognized that SPDEs can be used as an important modeling tool in various applied disciplines such as fluid mechanics, oceanography, temperature anomalies, finance, economics, biological and ecological systems; cf. \cite{ChowBook,LototskyRozovsky2017Book,RozovskyRozovsky2018Book}. On the other hand, the  literature on statistical inference for SPDEs is, relatively speaking, limited. We refer to the recent survey \cite{Cialenco2018} for an overview of the literature and existing methodologies on statistical inference for parabolic SPDEs. Most of the existing results are obtained within the so-called spectral approach, when it is assumed that one path of $N$ Fourier modes of the solution is observed continuously over a finite interval of time, in which case usually the statistical problems are addressed via maximum likelihood estimators (MLEs). Asymptotic properties of the estimators are studied in the large number of Fourier modes regime, $N\rightarrow\infty$, while time horizon is fixed. In particular, there are only few works related to Bayesian statistics for infinite dimensional evolution equations \cite{Bishwal2002,Bishwal1999,PrakasaRao2000}. As usual, studying SPDEs driven by multiplicative noise is more involved, and the parameter estimation problems for such equations are not an exception; the literature on this topic is also limited \cite{CialencoLototsky2009,Cialenco2010,PospivsilTribe2007,CialencoHuang2017,BibingerTrabs2017}.

The main goal of this paper is twofold: to study the parameter estimation problem for the drift coefficient of linear SPDEs driven by a \textit{multiplicative noise} (of special structure) and by applying the \textit{Bayesian estimation} procedure. Similar to the existing literature, we are assuming the spectral approach, and the main objective is to derive Bayesian type estimators and to study their asymptotic properties as $N\to\infty$. Besides contributing to these two important and undeveloped topics, the obtained results will prepare the foundation for studying similar problems for more complex (nonlinear) equations. We consider a multiplicative noise of special structure, which is customary considered in the SPDE applied literature. For example, such type of multiplicative noise appears in modeling and studying the dynamics of geophysical fluids \cite{Glatt-HoltzZiane2009,Glatt-HoltzZiane2008}, studying the stochastic primitive equations \cite{Glatt-HoltzEtAl2014} or  stochastically forced shell model of turbulent flow \cite{FriedlanderGlatt-HoltzVicol2016}, etc. Needless to say, this particular noise structure can be potentially used as modeling feature in SPDEs where the noise is not derived from the first principles, but rather added to capture the imperfections in model and/or measurements. This work is the first attempt to study parameter estimation problems for SPDEs driven by this multiplicative noise. The case of large time asymptotics is omitted here, since it is easily reduced to the corresponding statistical problem for finite dimensional stochastic differential equations, which is a well developed field. It is worth mentioning that the asymptotic properties of MLEs for these equations are trivially obtained, and we mention them here only briefly since they are used in derivation of convergence of proposed Bayesian estimators.

\medskip
\noindent The main contributions of this paper can be summarized as follows:
\vspace{-0.4em}
\begin{itemize}\addtolength{\itemsep}{-0.4\baselineskip}
  \item[$\diamond$]  We derive and study the asymptotic properties of MLE and Bayesian type estimators for the drift coefficient of a stochastic evolution system driven by a multiplicative (space-time) noise.
   \item[$\diamond$] We show that the considered statistical model is regular, and uniformly asymptotically normal, in the sense of \cite{IbragimovKhasminskiiBook1981}, and fits the classical local asymptotic normality (LAN) paradigm. In particular, under suitable assumptions, the MLE and the Bayesian estimators are weakly consistent, asymptotically normal, efficient, and asymptotically equivalent in the class of \textit{loss functions with polynomial growth}.
   \item[$\diamond$] We prove a Bernstein-Von Mises type result, that strengthens the existing results in the literature, and that also allows to investigate the Bayesian type estimators with respect to a larger class of priors and loss functions than that covered by classical asymptotic theory.
   \item[$\diamond$] We prove strong consistency and asymptotic normality of Bayesian estimators in the class of \textit{loss functions of at most exponential growth.}
\end{itemize}
The obtained results and developed techniques, besides their stand along merits, could be potentially useful for investigating some related problems, such as asymptotic properties of estimators in the simultaneous large times and large number of Fourier modes regime, discrete sampling, etc.

The paper is organized as follows. In Section \ref{sec:setup} we setup the problem and provide sufficient conditions on model parameters for the well-posedness of the solution of the underlying SPDEs.  Also here, we specify the statistical model, and show that the model is regular (in statistical sense) and uniformly asymptotically normal.  Section~\ref{sec:MLE} is devoted to MLE and its asymptotic properties. Using LAN approach we show that MLE is weakly consistent, asymptotically normal, and asymptotically efficient; see Theorem~\ref{thm:MLEUAN}. In addition, we also establish the strong consistency and asymptotic normality of MLE by exploiting the specific structure of the estimators. The Bayesian estimators are investigated in Section~\ref{sec:BayeEst}.  We start with the derivation of the Bayesian estimators and briefly cite their properties within the existing general inference theory; see Theorem~\ref{thm:BayeUANPoly}. Next, we proceed to one of the key results of this paper, Theorem \ref{thm:BvM}, a Bernstein-Von Mises type result.
Relevant asymptotic properties of the Bayesian estimators are proved in Sections~\ref{sec:AsymTildeBetaN} and \ref{sec:AsymHatBetaN}.
Using simulation technique we apply the obtained theoretical results to a stochastic heat equation; see Section \ref{sec:NumericalEx}. Some auxiliary results are deferred to Appendix~\ref{AppendA}. Finally, in Appendix \ref{AppendixB} we provide some reasoning on the form of the posterior density used in our framework.

\section{Preliminaries and setup of the problem}\label{sec:setup}

Let $(\Omega,\sF,\{\sF_{t}\}_{t\geq 0},\bP)$ be a stochastic basis satisfying the usual assumptions, on which we consider a sequence of independent standard Brownian motions $\{w_{k}\}_{k\in\bN}$. Assume that $H$ is a separable Hilbert space, with the corresponding inner product $(\cdot,\cdot)_{H}$. Let $A$ be a positive definite self-adjoint operator in $H$ that has only point spectrum, denoted by $\{\mu_{k}\}_{k\in\bN}$, with the corresponding eigenvectors $\{h_{k}\}_{k\in\bN}$. We make the standing assumption that $\{h_{k}\}_{k\in\bN}$ forms a complete orthonormal system in $H$, and $\mu_{k}\rightarrow\infty$. We will denote by $\{H^{\gamma},\gamma\in\bR\}$ the scale of Hilbert spaces generated by the operator $A$, i.e., $H^{\gamma}$ is equal to the closure of the collection of all finite linear combinations of $\{h_{k}\}_{k\in\bN}$ with respect to the norm $\|\cdot\|_{\gamma}:=\|A^{\gamma}\cdot\|_{H}$, with $A^{\gamma}v:=\sum_{k\in\bN}\mu_{k}^{\gamma}v_{k}h_{k}$ for $v=\sum_{k\in\bN}v_{k}h_{k}$.
We will also denote by $[\cdot,\cdot]$ the dual pair between $H^{1/2}$ and $H^{-1/2}$ relative to the inner product $(\cdot, \cdot)_H$.

We consider the following stochastic evolution equation
\begin{align}\label{eq:mainSPDE}
\begin{cases}
\dif u(t)+\theta Au(t)\dif t=\sigma\sum_{k=1}^{\infty}u_{k}(t)h_{k}q_{k}\dif w_{k}(t),\\
u(0)=u_{0}\in H,
\end{cases}
\end{align}
where $u_{k}(t):=(u(t),h_{k})_{H}$, $t\geq 0$, are the Fourier modes of the solution $u$ with respect to $\{h_{k}\}_{k\in\bN}$, $\theta,\sigma\in\bR_{+}:=(0,\infty)$, and $\{q_{k}\}_{k\in\bN}$ is a sequence in $\bR_{+}$.

The well-posedness of equation~\eqref{eq:mainSPDE} can be established either directly or by using some standard results from the general theory of linear SPDEs (see for instance \cite[Section~4.4]{LototskyRozovsky2017Book}).

\begin{theorem}\label{thm:EnU}
Let $T>0$, $u_0\in L_2(\Omega;H)$, and assume that there exist $N_0\in\bN$ and $c>0$ such that
\begin{align}\label{eq:mainAssumption}
2\theta-\sigma^2\frac{q_k^2}{\mu_k} \ge c, \quad \textrm{for all } k\ge N_0.
\end{align}
Then, equation~\eqref{eq:mainSPDE} admits a unique solution  $u\in L^2(\Omega;(C(0,T);H))\cap L^2(\Omega\times(0,T);H^{\frac{1}{2}})$, and
\begin{align*}
\bE\Big(\sup_{t\in(0,T)}\left\|u(t)\right\|_H^2 + \int_{0}^T\left\|u(t)\right\|_{1/2}^2 \dif t \Big) \le C\|u_0\|_{H}^2,
\end{align*}
for some constant $C$.
\end{theorem}
\begin{proof} In view of \cite[Theorem 4.4.3]{LototskyRozovsky2017Book},  it is enough to show that the following parabolicity condition holds true,
\begin{align}\label{eq:parabolicity}
-2\theta[Av,v] + \sum_{k=1}^\infty\sigma^2 q_k^2 v_k^2 + c_A\|v\|_{\frac12}^2 \le M\|v\|_H^2,\quad v\in H^{\frac12},
\end{align}
for some positive constant $c_A$ and $M$. Note that, $[Av,v]=\|v\|_{\frac12}^2= \sum_{k=1}^\infty \mu_k v_k^2$, and thus, the left hand side of~\eqref{eq:parabolicity} writes
\begin{align*}
\sum_{k=1}^\infty \left(-(2\theta+c_A) \mu_k + \sigma^2 q_k^2 \right)v_k^2 = -\sum_{k=1}^\infty \mu_k\left(2\theta-\sigma^2\frac{q_k^2}{\mu_k}-c_A\right)v_k^2.
\end{align*}
We put $c_A = c$, and by taking into account \eqref{eq:mainAssumption}, as well as the fact that $\mu_k\to\infty$, the condition~\eqref{eq:parabolicity} follows at once.
\end{proof}

\begin{remark}\label{rem:condExist} It is worth mentioning two simple examples of system's parameters that satisfy~\eqref{eq:mainAssumption}:
\begin{enumerate}[(E1)]
  \item there exist $\varepsilon>0$, and  $C>0$, such that $q_k^2\le C\mu_k^{1-\varepsilon}$, for sufficiently large $k$. %\label{eq:mainAssump1}\\
\item $2\theta - \sigma^2>0$, and $q_k^2\le\mu_k$, for sufficiently large $k$. %\label{eq:mainAssump2}
\end{enumerate}
\end{remark}

Although the equation \eqref{eq:mainSPDE} is driven by a multiplicative noise, due to the special structure of the noise, it is a diagonalizable SPDE, namely the Fourier modes of the solution satisfy an infinite dimensional system of decoupled equations
\begin{align}\label{eq:SDEuk}
\dif u_{k}(t)+\theta\mu_{k}u_{k}(t)\dif t=\sigma q_{k}u_{k}(t)\dif w_{k}(t),\quad t\in[0,T],\,\,k\in\bN,
\end{align}
with initial condition $u_{k}(0)=(u_{0},h_{k})_{H}$. Hence, we have that
\begin{align}\label{eq:Soluk}
u_{k}(t)=u_{k}(0)\exp\left(-\bigg(\theta\mu_{k}+\frac{1}{2}\sigma^{2}q_{k}^{2}\bigg)t+\sigma q_{k}w_{k}(t)\right),\quad t\in[0,T],\,\,k\in\bN.
\end{align}
Without loss of generality we will assume that $u_{k}(0)\neq 0$, for all $k\in\bN$. We note that, one could use directly the above form of $u_k$'s to prove the well-posedness of \eqref{eq:mainSPDE}.

\subsection{Statistical model and its basic properties}
We will take the continuous-time observation framework by assuming that the solution $u$, as an object in $H$, or in a finite dimensional projection of $H$, is observed continuously in time for all $t\in[0,T]$, and for some fixed horizon $T$. We assume that $\sigma$, and $q_{k}$, $k\in\bN$, are known constants, since generally speaking under the continuous-time sampling scheme, using quadratic variation arguments, these parameters can be found exactly. We will be interested in estimating the unknown parameter $\theta\in\bR_{+}$, with $\theta_{0}$ being the true value of this parameter of interest. In what follows we will denote by $u^{\theta}$ the solution to \eqref{eq:mainSPDE} that corresponds to the parameter $\theta$, and correspondingly, we put $u_{k}^{\theta}:=(u^{\theta},h_{k})_{H}$, $k\in\bN$. If no confusion arises, we will continue to write $u$ and $u_{k}$ instead of $u^{\theta_{0}}$ and $u^{\theta_{0}}_{k}$.

We will also assume that $q_{k}\neq 0$ for all $k\in\bN$. If $q_{k}=0$ for some $k\in\bN$, then $\theta$ can be found exactly, and the considered statistical problem becomes trivial.

In this study, we will assume that one path of the first $N$ Fourier modes $(u_{1}(t),\ldots,u_{N}(t))$ is observed continuously over a fixed time interval $[0,T]$, for some $T>0$. We will focus on the asymptotic properties in large number of Fourier modes, $N\rightarrow\infty$, while $T$ is fixed. The large time asymptotics $T\rightarrow\infty$, with $N$ fixed, reduces to existing results for finite dimensional systems of stochastic differential equations, which is well understood. The mixed case with both $N,T\rightarrow\infty$ is left for further studies.

We begin by placing the considered statistical model within classical  asymptotic theory of statistical estimation.
Let $C([0,T];\bR^{N})$ denote the space of all $\bR^{N}$-valued continuous functions on $[0,T]$. The cylindrical (Borel) $\sigma$-field on $C([0,T];\bR^{N})$ is denoted by $\cB(C([0,T];\bR^{N}))$. For every $\theta\in\bR_{+}$, let $\bP^{\theta}_{N}$ be the probability measure on $(C([0,T];\bR^{N}),\cB(C([0,T];\bR^{N}))$ induced by the projected solution $U^{\theta}_{N}:=\{(u^\theta_{1}(t),\ldots,u^{\theta}_{N}(t)),\,t\in[0,T]\}$. As before, we simply write $U_{N}$ and $\bP_{N}$ instead of $U_{N}^{\theta_{0}}$ and $\bP^{\theta_{0}}_{N}$. The measures $\bP^{\theta}_{N}$ and $\bP_{N}$ are equivalent, and the Likelihood Ratio, or the Radon-Nikodym derivative, is given by (cf. \cite[Section 7.6.4]{LiptserShiryayev2000})
\begin{align}\label{eq:LogLikelihood}
\frac{\dif\bP^{\theta}_{N}}{\dif\bP_{N}}({U}_{N})=\exp\bigg(\frac{\theta_{0}-\theta}{\sigma^{2}}\sum_{k=1}^{N}\mu_{k}q_{k}^{-2}\!\int_{0}^{T}\frac{\dif u_{k}(t)}{u_{k}(t)}+\frac{\big(\theta_{0}^{2}-\theta^{2}\big)T}{2\sigma^{2}}\sum_{k=1}^{N}\mu_{k}^{2}q_{k}^{-2}\bigg).
\end{align}
The Fisher information at $\theta_{0}$ is then given by
\begin{align}\label{eq:FisherInfo}
I_{N}=I_{N}(\theta_{0}):=\bE\left(\bigg(\frac{\partial}{\partial\theta}\ln\frac{\dif\bP^{\theta}_{N}}{\dif\bP_{N}}(U_{N})\bigg|_{\theta=\theta_{0}}\bigg)^{2}\right)=\frac{T}{\sigma^{2}}\sum_{k=1}^{N}\mu_{k}^{2}q_{k}^{-2},
\end{align}
which is, in particular, independent of $\theta_{0}$.

Next we will show that the statistical model $\sE:=\{C([0,T];\bR^{N}),\cB(C([0,T];\bR^{N})),\bP_{N}^{\theta},\theta\in\bR_{+}\}$ is regular; cf. \cite[Section I.7]{IbragimovKhasminskiiBook1981}.

\begin{theorem}\label{thm:Regularity}
The statistical model $\sE$ is regular in $\bR_{+}$.
\end{theorem}

\begin{proof} We will follow \cite[Section I.7]{IbragimovKhasminskiiBook1981}, and show that
\begin{itemize}
\item [(a)] $\dif\bP^{\theta}_{N}/\dif\bP_{N}(U_{N})$ is a continuous function of $\theta\in\bR_{+}$, $\bP$-almost surely;
\item [(b)] $\{\bP_{N}^{\theta},\theta\in\bR_{+}\}$ has finite Fisher information for each $\theta\in\bR_{+}$;
\item [(c)] The function $\psi(U_{N},\theta):=\partial\sqrt{\dif\bP^{\theta}_{N}/\dif\bP_{N}(U_{N})}/\partial\theta$ is continuous in $L^{2}(\Omega;\bP)$.
\end{itemize}
In view of \eqref{eq:LogLikelihood} and \eqref{eq:FisherInfo}, $\{\bP_{N}^{\theta},\theta\in\bR_{+}\}$ clearly satisfies properties (a) and (b).
To prove (c), by \eqref{eq:LogLikelihood} we first note that $\sqrt{\dif\bP^{\theta}_{N}/\dif\bP_{N}(U_{N})}$ is continuously differentiable with respect to $\theta$ in $\bR_{+}$, for every $\omega\in\Omega$, and
\begin{align*}
\psi(U_{N},\theta)=\bigg(-\frac{1}{2\sigma^{2}}\sum_{k=1}^{N}\mu_{k}q_{k}^{-2}\int_{0}^{T}\frac{\dif u_{k}(t)}{u_{k}(t)}-\frac{1}{2\sigma^{2}}\sum_{k=1}^{N}\mu_{k}^{2}q_{k}^{-2}\bigg)\sqrt{\frac{\dif\bP^{\theta}_{N}}{\dif\bP_{N}}({U}_{N})}
\end{align*}
is continuous with respect to $\theta$ in $\bR_{+}$ for every $\omega\in\Omega$. Moreover, for any $\theta,\theta'\in\bR_{+}$ with $|\theta-\theta'|\leq 1$,
\begin{align*}
&\big(\psi(U_{N},\theta)-\psi(U_{N},\theta')\big)^{2}\leq 2\big(\psi^{2}(U_{N},\theta)+\psi^{2}(U_{N},\theta')\big)\\
&\quad\leq\bigg(-\frac{1}{\sigma^2}\sum_{k=1}^{N}\frac{\mu_{k}}{q_{k}^{2}}\int_{0}^{T}\frac{\dif u_{k}(t)}{u_{k}(t)}-\frac{1}{\sigma^2}\sum_{k=1}^{N}\frac{\mu_{k}^{2}}{q_{k}^{2}}\bigg)^{2}\exp\bigg(\frac{\theta+1}{\sigma^{2}}\sum_{k=1}^{N}\frac{\mu_{k}}{q_{k}^{2}}\left|\int_{0}^{T}\frac{\dif u_{k}(t)}{u_{k}(t)}\right|-\frac{\theta_{0}^{2}T}{2\sigma^{2}}\sum_{k=1}^{N}\frac{\mu_{k}^{2}}{q_{k}^{2}}\bigg),
\end{align*}
which is integrable with respect to $\bP$ in view of \eqref{eq:SDEuk}. Therefore, by the dominated convergence theorem,
\begin{align*}
\lim_{\theta'\rightarrow\theta}\bE\Big(\big(\psi(U_{N},\theta)-\psi(U_{N},\theta')\big)^{2}\Big)=0,
\end{align*}
which completes the proof of property (c), and thus concludes the proof of the theorem.
\end{proof}

It turns out that the statistical model  $\sE$, being regular, fits also nicely in the general framework of \textit{Local Asymptotic Normality} (LAN).
As next result shows,  $\sE$ is actually \textit{uniformly asymptotically normal}; cf. \cite[Definition II.2.2]{IbragimovKhasminskiiBook1981}).
This fundamental property will allow to show that maximum likelihood estimator and Bayesian estimators for $\theta$ are not only consistent and asymptotically normal, but also asymptotically efficient and asymptotically equivalent. See Section~\ref{sec:MLE}, Section~\ref{sec:BayeEst} and discussions therein on these theoretical aspects, as well as the comparison between them and those developed and proposed in this paper without using LAN framework.

\begin{theorem}\label{thm:UAN}
The family $\{\bP_{N}^{\theta},\theta\in\bR_{+}\}$ is uniformly asymptotically normal.
\end{theorem}

\begin{proof} For reader's convenience, we first match the notations in \cite[Definition II.2.2]{IbragimovKhasminskiiBook1981} with those from our model: the perturbation variable $\varepsilon_N$ will be $1/N$, and $\varphi(\varepsilon_N,t)=1/\sqrt{I_{N}}$. For any sequences $\{\vartheta_{N}\}_{N\in\bN}\subset\bR_{+}$ and $\{\eta_{N}\}_{N\in\bN}\subset\bR$ such that $\eta_{N}\rightarrow\eta\in\bR$, as $N\rightarrow\infty$, and that $\theta_{N}:=\vartheta_{N}+\eta_{N}/\sqrt{I_{N}}\in\bR_{+}$ for all $N\in\bN$, by \eqref{eq:LogLikelihood} and \eqref{eq:FisherInfo} we have
\begin{align*}
\frac{\dif\bP^{\theta_{N}}_{N}}{\dif\bP^{\vartheta_{N}}_{N}}\big(U_{N}^{\vartheta_{N}}\big)&=\exp\left(\frac{\vartheta_{N}-\theta_{N}}{\sigma^{2}}\sum_{k=1}^{N}\mu_{k}q_{k}^{-2}\int_{0}^{T}\frac{\dif u_{k}^{\vartheta}(t)}{u_{k}^{\vartheta}(t)}+\frac{\big(\vartheta_{N}^{2}-\theta_{N}^{2}\big)T}{2\sigma^{2}}\sum_{k=1}^{N}\mu_{k}q_{k}^{-2}\right)\\
%&=\exp\left(\frac{-\eta_{N}}{\sigma^{2}\sqrt{I_{N}}}\sum_{k=1}^{N}\mu_{k}q_{k}^{-2}\int_{0}^{T}\frac{\dif u_{k}^{\vartheta_{N}}(t)}{u_{k}^{\vartheta_{N}}(t)}-\frac{\big(\eta_{N}^{2}+2\vartheta_{N}\eta_{N}\sqrt{I_{N}}\big)T}{2\sigma^{2}I_{N}}\sum_{k=1}^{N}\mu_{k}^{2}q_{k}^{-2}\right)\\
&=\exp\bigg(\frac{-\eta_{N}}{\sigma^{2}\sqrt{I_{N}}}\sum_{k=1}^{N}\mu_{k}q_{k}^{-2}\int_{0}^{T}\frac{\dif u_{k}^{\vartheta_{N}}(t)}{u_{k}^{\vartheta_{N}}(t)}-\frac{\eta_{N}^{2}}{2}-\vartheta_{N}\eta_{N}\sqrt{I_{N}}\bigg)\\
&=\exp\left(-\eta_{N}\bigg(\frac{1}{\sigma^{2}\sqrt{I_{N}}}\sum_{k=1}^{N}\mu_{k}q_{k}^{-2}\int_{0}^{T}\frac{\dif u_{k}^{\vartheta_{N}}(t)}{u_{k}^{\vartheta_{N}}(t)}+\vartheta_{N}\sqrt{I_{N}}\bigg)-\frac{\eta^{2}}{2}+\bigg(\frac{\eta^{2}}{2}-\frac{\eta_{N}^{2}}{2}\bigg)\right).
\end{align*}
Clearly $\eta^{2}/2-\eta_{N}^{2}/2\rightarrow 0$, as $N\rightarrow\infty$. It remains to show that $\eta_{N}\xi_{N}$ converges in distribution to some centered normal distribution, as $N\rightarrow\infty$, where
\begin{align}\label{eq:xiN}
\xi_{N}:= -\frac{1}{\sigma^{2}\sqrt{I_{N}}}\sum_{k=1}^{N}\mu_{k}q_{k}^{-2}\int_{0}^{T}\frac{\dif u_{k}^{\theta}(t)}{u_{k}^{\theta}(t)}+\theta\sqrt{I_{N}}= -\frac{1}{\sigma\sqrt{I_{N}}}\sum_{k=1}^{N}u_{k}q_{k}^{-1}w_{k}(T).
\end{align}
%In particular, $\xi_{N}$ is independent of $\theta\in\bR_{+}$. By the independence of $\{w_{k}\}_{k\in\bN}$,
Clearly, $\xi_{N}\ed\cN(0,1)$ under $\bP$, and therefore, by Slutsky's theorem, $\eta_{N}\xi_{N}\cd\cN(0,\eta^{2})$ under $\bP$, as $N\rightarrow\infty$, which completes the proof.
\end{proof}

\section{Maximum Likelihood Estimator}\label{sec:MLE}

In this section, we will investigate the asymptotic properties of the \textit{maximum likelihood estimator}  for the unknown parameter $\theta$. The obtained results, albeit simple, are important on their own, and we will also use them later to study the Bayesian estimators in Section~\ref{sec:BayeEst}.

By maximizing the likelihood ratio in \eqref{eq:LogLikelihood} with respect to $\theta$, we obtain the following estimator for $\theta$
\begin{equation}\label{eq:MLEIni1}
\wh{\theta}_{N}:=-\frac{\sum_{k=1}^{N}\mu_{k}q_{k}^{-2}\int_{0}^{T}\frac{\dif u_{k}(t)}{u_{k}(t)}}{T\sum_{k=1}^{N}\mu_{k}^{2}q_{k}^{-2}}.
\end{equation}
Using \eqref{eq:SDEuk}, by It\^o's formula, the estimator $\wh\theta_N$ can also be  written as
\begin{equation}\label{eq:MLEIni2}
\wh{\theta}_{N}:=-\frac{\sum_{k=1}^{N}\mu_{k}\big(q_{k}^{-2}\log(u_k(T)/u_k(0)) + \frac{1}{2} \sigma^2T\big)}{T\sum_{k=1}^{N}\mu_{k}^{2}q_{k}^{-2}},
\end{equation}
which can be useful for practical purposes.

Since $\theta\in\bR_{+}$, formally, the MLE for $\theta$ is given by
\begin{align}\label{eq:MLE}
\wh{\theta}_{N}^{\,\text{MLE}}:=\1_{\{\wh{\theta}_{N}\in\bR_{+}\}}\wh{\theta}_{N}.
\end{align}

Before we start the analysis on the asymptotic properties of the MLE, let us first introduce the following classes of loss functions, which will be used in the statement of asymptotic efficiency of MLE,  as well as in the later discussions on Bayesian estimators.

\begin{definition}\label{def:LossFunt}
Let $\mathbf{W}$ be the set of Borel measurable \textit{loss functions} $\ell:\bR\rightarrow[0,\infty)$ such that
\begin{itemize}
\item [(i)] $\ell$ is symmetric on $\bR$ and is non-decreasing on $[0,\infty)$;
\item [(ii)] $\ell(0)=0$ and $\ell$ is continuous at $x=0$, but is not identically 0.
\end{itemize}
Denote by $\mathbf{W}_{e,2}$ the set of functions $\ell\in\mathbf{W}$ whose growth as $|x|\rightarrow\infty$ is bounded by one of the functions $e^{c|x|^{r}}$ with $c\in\bR_{+}$ and $r\in(0,2)$, i.e. loss functions with `exponential' growth. The set of functions $\ell\in\mathbf{W}$ which possess a polynomial majorant as $|x|\rightarrow\infty$ will be denoted by $\mathbf{W}_{p}$.

We will denote by $\mathbf{W}'$ the class of loss functions $\ell$  that are non-negative, Borel-measurable, and locally bounded function on $\bR$, with $\ell(0)=0$, such that
\begin{align}\label{eq:MonotoneLossFunt}
\ell(x_{1})\leq\ell(x_{2}),\quad\text{for any }\,x_{1},x_{2}\in\bR\,\,\,\text{with}\,\,\,|x_{1}|\leq|x_{2}|.
\end{align}
Clearly, $\mathbf{W}\subset \mathbf{W}'$.
\end{definition}

The following result is a consequence of \cite[Theorem III.1.1, Theorem III.1.2, Corollary III.1.1]{IbragimovKhasminskiiBook1981}, which summaries the uniform weak consistency, the uniform asymptotic normality, and the asymptotic efficiency of $\wh{\theta}_{N}^{\,\text{MLE}}$.

\begin{theorem}\label{thm:MLEUAN}
Let $B\subset\bR_{+}$ be any compact set. Then, the following asymptotic properties of $\wh{\theta}^{\,\text{MLE}}_{N}$ hold true.
\begin{itemize}
\item[(a)] $\wh{\theta}^{\,\text{MLE}}_{N}$ is weakly consistent, uniformly in $\theta_{0}\in B$. That is, for any $\varepsilon>0$,
    \begin{align*}
    \lim_{N\rightarrow\infty}\sup_{\theta_{0}\in B}\bP\Big(\big|\wh{\theta}^{\,\text{MLE}}_{N}-\theta_{0}\big|>\varepsilon\Big)=0,
    \end{align*}
\item[(b)] $\wh{\theta}^{\,\text{MLE}}_{N}$ is asymptotically normal with parameter $(\theta_{0},I_{N}^{-1})$, uniformly in $\theta_{0}\in B$, where $I_{N}$ is given by \eqref{eq:FisherInfo}. That is, uniformly in $\theta_{0}\in B$,
    \begin{align*}
    \sqrt{I_{N}}\big(\wh{\theta}^{\,\text{MLE}}_{N}-\theta_{0}\big)\cd\cN(0,1),\quad N\rightarrow\infty.
    \end{align*}
\item[(c)] For any $\varepsilon>0$,
    \begin{align*}
    \lim_{N\rightarrow\infty}\sup_{\theta_{0}\in B}\bP\Big(\Big|\sqrt{I_{N}}\big(\wh{\theta}^{\,\text{MLE}}_{N}-\theta_{0}\big)-\xi_{N}\Big|>\varepsilon\Big)=0,
    \end{align*}
    where $\xi_{N}$ is defined as in \eqref{eq:xiN}.
\item[(d)] For any $\varpi\in\mathbf{W}_{p}$, $\wh{\theta}^{\,\text{MLE}}_{N}$ is asymptoticly efficient in $B$ for the loss function $\varpi_{N}(x):=\varpi(\sqrt{I_{N}}x)$. That is, for any $\theta_{0}\in B$,
    \begin{align*}
    \lim_{\delta\rightarrow 0+}\liminf_{N\rightarrow\infty}\sup_{|\theta-\theta_{0}|<\delta}\bE\Big(\varpi_{N}\big(\wh{\theta}_{N}^{\,\text{MLE}}-\theta\big)\Big)=\bE(\varpi(\xi)),
    \end{align*}
    where $\xi\ed\cN(0,1)$ under $\bP$.
\end{itemize}
\end{theorem}

\begin{proof}
We only need to check the conditions (N1)$-$(N4) in \cite[Section III.1]{IbragimovKhasminskiiBook1981}. The condition (N1) follows from Theorem \ref{thm:UAN}, while the condition (N2) is trivial since $\varphi(\varepsilon,t)=1/\sqrt{I_{N}}$ in our case. To verify the condition (N3), for any $\theta_{0}\in B$ and $\eta,\zeta\in\bR$ such that $\eta_{N}:=\theta_{0}+\eta/\sqrt{I_{N}}\in\bR_{+}$ and $\zeta_{N}:=\theta_{0}+\zeta/\sqrt{I_{N}}\in\bR_{+}$, by \eqref{eq:LogLikelihood} and \eqref{eq:xiN} we have
\begin{align*}
\bE\left(\left(\sqrt{\frac{\dif\bP^{\eta_{N}}_{N}}{\dif\bP_{N}}(U_{N})}-\sqrt{\frac{\dif\bP^{\zeta_{N}}_{N}}{\dif\bP_{N}}(U_{N})}\right)^{2}\right)=\bE\bigg(\Big(e^{-\eta\xi_{N}(\theta_{0})/2-\eta^{2}/4}-e^{-\zeta\xi_{N}(\theta_{0})/2-\zeta^{2}/4}\Big)^{2}\bigg).
\end{align*}
Similar to the proof of Theorem~\ref{thm:UAN}, we note that $\xi_{N}(\theta_{0})\ed\cN(0,1)$ under $\bP$, and thus
\begin{align*}
(\eta-\zeta)^{-2}\,\bE\left(\left(\sqrt{\frac{\dif\bP^{\eta_{N}}_{N}}{\dif\bP_{N}}(U_{N})}-\sqrt{\frac{\dif\bP^{\zeta_{N}}_{N}}{\dif\bP_{N}}(U_{N})}\right)^{2}\right)=\frac{2-2e^{-(\eta-\zeta)^{2}/8}}{(\eta-\zeta)^{2}}\leq 16.
\end{align*}
Therefore, the condition (N3) is valid with $\beta=m=2$ and any positive constants $B$, $R$, and $\alpha$. Similarly, for any $\theta_{0}\in B$, $n,N\in\bN$, and $\eta\in\bR$ with $\eta_{N}:=\theta_{0}+\eta/\sqrt{I_{N}}\in\bR_{+}$,
\begin{align*}
|\eta|^{n}\,\bE\left(\sqrt{\frac{\dif\bP^{\eta_{N}}_{N}}{\dif\bP_{N}}(U_{N})}\right)=|\eta|^{n}\,\bE\Big(e^{-\eta\xi_{N}(\theta_{0})/2-\eta^{2}/4}\Big)\leq\sup_{\eta\in\bR_{+}}|\eta|^{n}e^{-\eta^{2}/8}<\infty.
\end{align*}
which verifies the validity of the condition (N4). The proof is complete.
\end{proof}

Given the particular form of the estimator $\wh\theta_N$, one can establish its strong consistency and asymptotic normality. Indeed, by \eqref{eq:SDEuk}, \eqref{eq:MLEIni1} can be conveniently written as
\begin{align}\label{eq:MLEIni3}
\wh{\theta}_{N}=\theta_{0}-\frac{\sigma}{T}\frac{\sum_{k=1}^{N}\mu_{k}q_{k}^{-1}w_{k}(T)}{\sum_{k=1}^{N}\mu_{k}^{2}q_{k}^{-2}}=\theta_{0}-\frac{\sigma}{T}\frac{\sum_{k=1}^{N}a_{k}}{b_{N}},
\end{align}
where $a_{k}:=\mu_{n}q_{n}^{-1}w_{n}(T)$ and $b_{n}:=\sum_{k=1}^{n}\mu_{k}^{2}q_{k}^{-2}$, $n\in\bN$. Clearly, $b_{N}\nearrow\infty$ as $N\rightarrow\infty$. Moreover, in view of Lemma \ref{lem:techlem1}, we have that
\begin{align*}
\sum_{k=1}^{\infty}\frac{\text{Var}(a_{k})}{b_{k}^{2}}<\infty.
\end{align*}
By the Law of Large Numbers, Theorem~\ref{thm:SLLN}, $\sum_{k=1}^{N}a_{k}/b_{N}\rightarrow 0$ as $N\rightarrow\infty$, and thus by \eqref{eq:MLEIni3}, we obtain that
\begin{align}\label{eq:MLEIniStrConst}
\lim_{N\rightarrow\infty}\wh{\theta}_{N}=\theta_{0},\quad\bP-\text{a.s.},
\end{align}
namely $\wh{\theta}_{N}$ is a \textit{strongly consistent} estimator of $\theta_{0}$. Together with \eqref{eq:MLE}, we also have that
\begin{align*}
\lim_{N\rightarrow\infty}\wh{\theta}_{N}^{\,\text{MLE}}=\theta_{0},\quad\bP-\text{a.s.},
\end{align*}
that is $\wh{\theta}_{N}^{\,\text{MLE}}$ is a strongly consist estimator of $\theta_{0}$. In addition, by \eqref{eq:FisherInfo},  \eqref{eq:xiN}, and \eqref{eq:MLEIni3}, under $\bP$,
\begin{align}\label{eq:MLEIniAsymNormal}
\sqrt{I_{N}}\big(\wh{\theta}_{N}-\theta_{0}\big)=\xi_{N}\ed\cN(0,1),\quad\text{for any }N\in\bN\,\,\,\text{and}\,\,\,\theta_{0}\in\bR_{+}.
\end{align}

Finally, we conclude this section by showing that the estimator $\wh\theta_N$ is also asymptotically efficient in the class of loss functions with polynomial growth.
\begin{proposition}\label{cor:MLEIniUAN}
The properties (a)$-$(d) in Theorem \ref{thm:MLEUAN} hold true with $\wh{\theta}_{N}^{\,\text{MLE}}$ replaced by $\wh{\theta}_{N}$.
\end{proposition}

\begin{proof}
The properties (a)$-$(c) for $\wh{\theta}_{N}$ follow trivially from \eqref{eq:MLEIniAsymNormal}. The property (d) for $\wh{\theta}_{N}$ follows from \eqref{eq:MLEIniAsymNormal} and \cite[Theorem III.1.3]{IbragimovKhasminskiiBook1981}.
\end{proof}

\section{Bayesian Estimators}\label{sec:BayeEst}

In this section, we propose two Bayesian type estimators for $\theta$ with respect to scaled and unscaled loss functions, respectively, and we will investigate their asymptotic behavior. First, by similarity to the MLE estimator from previous section and following \cite[Section III.2]{IbragimovKhasminskiiBook1981}, we will study the asymptotic properties of the Bayesian estimator with respect to a scaled loss function $\ell_{N}(\cdot):=\ell(\sqrt{I_{N}}\,\cdot)$, where $\ell$ has a polynomial majorant. However, most of this section will be devoted to results beyond and independent of the classical asymptotic theory.
As custom for Bayesian statistics (cf. \cite{BorwankerKallianpurPrakasaRao1971} and \cite{Bishwal2002}), we prove a Bernstein-Von Mises theorem for the posterior density in which both the prior density and the test function have at most exponential growth rates. Consequently, this allows to study the asymptotic properties of Bayesian estimators with respect to any loss function $\ell$ that has an exponential majorant (in contrast to classical theory that allows only polynomial growth). Moreover, we show that the MLE and the proposed Bayesian estimators are asymptotically equivalent.

We begin with the definition of considered prior densities on $\bR_{+}$.

\begin{definition}\label{def:PriorDist}
Let $\cQ$ be the set of all non-negative and non-trivial functions on $\bR_{+}$. Denote by $\cQ_{e,2}$ the set of functions $\varrho\in\cQ$ which are continuous and positive on $\bR_{+}$ and whose growth as $\theta\rightarrow\infty$ is bounded by one of the functions $e^{c\,\theta^{r}}$ with $c\in\bR_{+}$ and $r\in(0,2)$. The set of functions $\varrho\in\cQ$ which are continuous and positive on $\bR_{+}$ and possess a polynomial majorant as $\theta\rightarrow\infty$ will be denoted by $\cQ_{p}$.
\end{definition}

Let $\varrho\in\cQ$. We define the posterior density as
\begin{align}\label{eq:PosteriorTheta}
p(\theta\,|\,U_{N}):=\frac{\displaystyle{\frac{\dif\bP^{\theta}_{N}}{\dif\bP_{N}}(U_{N})\varrho(\theta)}}{\displaystyle{\int_{\bR_{+}}\frac{\dif\bP^{\eta}_{N}}{\dif\bP_{N}}(U_{N})\varrho(\eta)\,\dif\eta}},\quad\theta\in\bR_{+}.
\end{align}
Informally, as standard in Bayesian inference, we took\footnote{As usual in statistics, the symbol $\varpropto$ will be used to denote equality between two quantities up to a normalized constant.}  `posterior $\varpropto$ likelihood $\times$ prior.'  For a more formal discussion of the rational behind \eqref{eq:PosteriorTheta} we refer to Appendix~\ref{AppendixB}.

Together with \eqref{eq:LogLikelihood}, \eqref{eq:FisherInfo}, and \eqref{eq:MLEIni1}, we deduce that
\begin{align}
p(\theta\,|\,U_{N})&=\frac{\displaystyle{\exp\bigg(\frac{\theta_{0}-\theta}{\sigma^{2}}\sum_{k=1}^{N}\mu_{k}q_{k}^{-2}\int_{0}^{T}\frac{\dif u_{k}(t)}{u_{k}(t)}+\frac{\big(\theta_{0}^{2}-\theta^{2}\big)T}{2\sigma^{2}}\sum_{k=1}^{N}\mu_{k}^{2}q_{k}^{-2}\bigg)\varrho(\theta)}}{\displaystyle{\int_{\bR_{+}}\exp\bigg(\frac{\theta_{0}-\eta}{\sigma^{2}}\sum_{k=1}^{N}\mu_{k}q_{k}^{-2}\int_{0}^{T}\frac{\dif u_{k}(t)}{u_{k}(t)}+\frac{\big(\theta_{0}^{2}-\eta^{2}\big)T}{2\sigma^{2}}\sum_{k=1}^{N}\mu_{k}^{2}q_{k}^{-2}\bigg)\varrho(\eta)\,\dif\eta}}\nonumber\\
&=\frac{\displaystyle{\exp\bigg(-\frac{\theta}{\sigma^{2}}\sum_{k=1}^{N}\mu_{k}q_{k}^{-2}\int_{0}^{T}\frac{\dif u_{k}(t)}{u_{k}(t)}-\frac{\theta^{2}T}{2\sigma^{2}}\sum_{k=1}^{N}\mu_{k}^{2}q_{k}^{-2}\bigg)\varrho(\theta)}}{\displaystyle{\int_{\bR_{+}}\exp\bigg(-\frac{\eta}{\sigma^{2}}\sum_{k=1}^{N}\mu_{k}q_{k}^{-2}\int_{0}^{T}\frac{\dif u_{k}(t)}{u_{k}(t)}-\frac{\eta^{2}T}{2\sigma^{2}}\sum_{k=1}^{N}\mu_{k}^{2}q_{k}^{-2}\bigg)\varrho(\eta)\,\dif\eta}}\nonumber\\
\label{eq:PosteriorThetaDeri} &=\frac{e^{-I_{N}(\theta-\wh{\theta}_{N})^{2}/2}\,\varrho(\theta)}{\displaystyle{\int_{\bR_{+}}e^{-I_{N}(\eta-\wh{\theta}_{N})^{2}/2}\,\varrho(\eta)\,\dif\eta}}.
\end{align}

\begin{remark}\label{rem:ConjugatePrior} Clearly both $\cQ_{e,2}$ and $\cQ_p$ are large classes of functions. One class of (conjugate) priors worth mentioning is the truncated normal $\cN(\mu_0,\sigma^{2}_0;\bR_{+})$ supported on $\bR_{+}$. Let the prior $\varrho$ be the density of the truncated normal $\cN(\mu_0,\sigma_0^{2};\bR_{+})$, namely $\varrho(\theta) \varpropto \phi((\theta-\mu_0)/\sigma_0), \ \theta_0>0$, where $\phi$ is the density of a standard Gaussian random variable. By \eqref{eq:PosteriorThetaDeri}, for any $\theta\in\bR_{+}$, we have that
\begin{align*}
p(\theta\,|\,U_{N})&\propto  \exp\bigg(-\frac{\big(\theta-\wh{\theta}_{N}\big)^{2}}{2I_{N}^{-1}}-\frac{(\theta-\mu_0)^{2}}{2\sigma_0^{2}}\bigg)\propto\exp\left(-\frac{\Big(\theta-\big(\sigma_0^{2}\,\wh{\theta}_{N}+I_{N}^{-1}\mu_0\big)\big/\big(\sigma_0^{2}+I_{N}^{-1}\big)\Big)^{2}}{2\sigma_0^{2}I_{N}^{-1}\big/\big(\sigma_0^{2}+I_{N}^{-1}\big)}\right),
\end{align*}
namely $p(\theta\,|\,U_{N})$ is a truncated normal $ \cN((\sigma_0^{2}\wh{\theta}_{N}+I_{N}^{-1}\mu_0)/(\sigma_0^{2}+I^{-1}_{N}),(\sigma_0^{2}I^{-1}_{N})/(\sigma_0^{2}+I^{-1}_{N});\bR_{+}) $.

Of course, one can take uninformative prior $\varrho(\theta)\propto 1$. Although such prior has no probabilistic meaning, the corresponding posterior is well-defined and preserves all convergence results listed below.
\end{remark}

We now introduce two Bayesian type estimators, one with respect to a loss function $\ell\in\mathbf{W}'$ and one with respect to its scaled version.
A Bayesian estimator with respect to $\ell\in\mathbf{W}'$ is defined as
\begin{align}\label{eq:DefWHBetaNW}
\wh{\beta}_{N}:=\argmin_{\beta\in\bR_{+}}\int_{\bR_{+}}\ell(\eta-\beta)\,p(\eta\,|\,U_{N})\,\dif\eta,
\end{align}
given that the minimum is strict and attainable. In line with \cite[Section III.2]{IbragimovKhasminskiiBook1981}, we define a Bayesian estimator with respect to $\ell_{N}(x):=\ell(\sqrt{I_{N}}x)$, where $\ell\in\mathbf{W}$, as
\begin{align}\label{eq:DefWTBetaNeta}
\wt{\beta}_{N}:=\argmin_{\beta\in\bR_{+}}\int_{\bR_{+}}\ell\big(\sqrt{I_{N}}(\eta-\beta)\big)\,p(\eta\,|\,U_{N})\,\dif\eta,
\end{align}
given that the minimum is strict and attainable.

Before proceeding with asymptotic properties of these estimators, several comments pinpointing the intuition behind these definitions are in order.
\begin{itemize}
\item [(a)] Let $U^{\Theta}_{N}$ be the $\bR^{N}$-valued process obtained by substituting $\theta$ with $\Theta$ in $U_{N}^{\theta}$, where $\Theta:(\Omega,\sF,\bP)\rightarrow(\bR_{+},\cB(\bR_{+})$ admits a proper prior density $\varrho$ and is independent of the Brownian motions $\{w_{k},k\in\bN\}$. Recall that the Bayesian risk of an estimator $\overline{\theta}_{N}\in\sigma(U^{\Theta}_{N})$ with respect to a loss function $\ell$ is defined as
    \begin{align*}
    r\big(\overline{\theta}_{N};\ell\big):=\bE\big(\ell\big(\overline{\theta}_{N}-\Theta\big)\big)=\bE\Big(\bE\Big(\ell\big(\overline{\theta}_{N}-\Theta\big)\Big|\,U_{N}^{\Theta}\Big)\Big)=\bE\bigg(\int_{\bR}\ell\big(\overline{\theta}_{N}-\eta\big)p(\eta\,|\,U^{\Theta}_{N})\dif\eta\bigg),
    \end{align*}
    where the last equality follows from \eqref{eq:CondThetaUNTheta}. Hence, $\wh{\beta}_{N}$ is the minimizer of the Bayesian risk $r(\,\cdot\,;\ell)$, while $\wt{\beta}_{N}$ is the minimizer of the Bayesian risk $r(\,\cdot\,;\ell_{N})$.
\item [(b)] The definition of the Bayesian estimator $\wh{\beta}_{N}$ is standard, in which the loss function does not depend on the sample size $N$, and is therefore more accessible for computational purposes. On the other hand, as argued in \cite{IbragimovKhasminskiiBook1981}, using a scaled loss function in defining a Bayesian estimator is more appropriate for analyzing some asymptotic properties of the estimator. Indeed, recall that
    (cf. \cite[Definition I.9.1]{IbragimovKhasminskiiBook1981}) a sequence of estimators $\{\overline{\theta}_{N}\}_{N\in\bN}$ of $\theta_{0}$ is called asymptotically efficient\footnote{The asymptotic efficiency studied in part (d) of Theorems \ref{thm:MLEUAN} and \ref{thm:BayeUANPoly} is a special case under the framework of local asymptotic normality.} in $\bR_{+}$ with respect to a sequence of loss function $\{\varpi_{N}\}_{N\in\bN}$ if, for any open set $O\subset\bR_{+}$ and any estimator $\wt{\theta}_{N}$ of $\theta_{0}$,
    \begin{align*}
    \lim_{N\rightarrow\infty}\bigg(\inf_{\wt{\theta}_{N}}\sup_{\theta_{0}\in O}\bE\Big(\varpi_{N}\big(\wt{\theta}_{N}-\theta_{0}\big)\Big)-\sup_{\theta_{0}\in O}\bE\Big(\varpi_{N}\big(\overline{\theta}_{N}-\theta_{0}\big)\Big)\bigg)=0.
    \end{align*}
    As mentioned in \cite[Section I.9]{IbragimovKhasminskiiBook1981}, the loss function $\varpi_{N}$ should depend on $N$ (and typically $\varpi_{N}(x)=\varpi(\sqrt{I_{N}}x)$) in order to capture more subtle difference between estimators. Hence, it is natural to scale the loss function in the definition of $\wt{\beta}_{N}$ in the same way as $\varpi_{N}$ when investigating the asymptotic efficiency of $\wt{\beta}_{N}$.
\item [(c)] When $\ell(x)=|x|^{\alpha}$ for some $\alpha>0$, $\wt{\beta}_{N}$ coincides with $\wh{\beta}_{N}$. Moreover, when $\ell\in\mathbf{W}_{p}$, $\wt{\beta}_{N}$ and $\wh{\beta}_{N}$ are expected to have the same asymptotic behavior as $N\rightarrow\infty$. Clearly, this may not be the case if $\ell\in\mathbf{W}_{e,2}$, or more generally $\ell\in\mathbf{W}'$, yet another reason to distinguish the two proposed estimators.
\end{itemize}

The following theorem summarizes the asymptotic properties of $\wt{\beta}_{N}$ with $\varrho$, and $\ell$ having polynomial growth, which is a direct consequence of \cite[Theorem III.2.1, Theorem III.2.2]{IbragimovKhasminskiiBook1981}.

\begin{theorem}\label{thm:BayeUANPoly}
Let $B\subset\bR_{+}$ be any compact set, $\varrho\in\cQ_{p}$, and $\ell\in\mathbf{W}_{p}$. Then, the following asymptotic properties hold for $\wt{\beta}_{N}$.
\begin{itemize}
\item[(a)] $\wt{\beta}_{N}$ is weakly consistent, uniformly in $\theta_{0}\in B$. That is, for any $\varepsilon>0$,
    \begin{align*}
    \lim_{N\rightarrow\infty}\sup_{\theta_{0}\in B}\bP\Big(\big|\wt{\beta}_{N}-\theta_{0}\big|>\varepsilon\Big)=0,
    \end{align*}
\item[(b)] $\wt{\beta}_{N}$ is asymptotically normal with parameter $(\theta_{0},I_{N}^{-1})$, uniformly in $\theta_{0}\in B$, where $I_{N}$ is given by \eqref{eq:FisherInfo}. That is, uniformly in $\theta_{0}\in B$,
    \begin{align*}
    \sqrt{I_{N}}\big(\wt{\beta}_{N}-\theta_{0}\big)\cd\cN(0,1),\quad N\rightarrow\infty.
    \end{align*}
\item[(c)] For any $\varepsilon>0$,
    \begin{align*}
    \lim_{N\rightarrow\infty}\sup_{\theta_{0}\in B}\bP\Big(\Big|\sqrt{I_{N}}\big(\wt{\beta}_{N}-\theta_{0}\big)-\xi_{N}\Big|>\varepsilon\Big)=0,
    \end{align*}
    where $\xi_{N}$ is defined as in \eqref{eq:xiN}.
\item[(d)] For any $\varpi\in\mathbf{W}_{p}$, $\wt{\beta}_{N}$ is asymptotically efficient in $B$ for the loss function $\varpi_{N}(x):=\varpi(\sqrt{I_{N}}x)$. That is, for any $\theta_{0}\in B$,
    \begin{align*}
    \lim_{\delta\rightarrow 0+}\liminf_{N\rightarrow\infty}\sup_{|\theta-\theta_{0}|<\delta}\bE\Big(\varpi_{N}\big(\wt{\beta}_{N}-\theta\big)\Big)=\bE(\varpi(\xi)),
    \end{align*}
    where $\xi\ed\cN(0,1)$ under $\bP$.
\end{itemize}
\end{theorem}

\begin{proof}
It is sufficient to check the conditions (N1)$-$(N4) in \cite[Section III.1]{IbragimovKhasminskiiBook1981}, which has been verified in the proof of Theorem \ref{thm:MLEUAN}.
\end{proof}

With the help of the Bernstein-Von Mises theorem from next section, we will be able to investigate the asymptotic properties of both $\wt{\beta}_{N}$ and $\wh{\beta}_{N}$ with a set of priors and loss functions having exponential growth rates, as shown in Sections \ref{sec:AsymTildeBetaN} and \ref{sec:AsymHatBetaN} below, respectively.

\subsection{The Bernstein-Von Mises Theorem}\label{sec:BvM}

We recall that, generally speaking, the Bernstein-von Mises type theorem states that the posterior distribution of the normalized distance between the randomized parameter $\Theta$ and $\wh{\theta}_{N}$ (given as in \eqref{eq:MLEIni1}) is asymptotically normal. This type of result implies that the posterior distribution measure approaches the Dirac measure as the number of observations increases; see Remark \ref{rem:BvM} below for more details. Moreover, it also serves as an essential tool in derivation of some asymptotic properties of Bayesian estimators.
To develop the Bernstein-Von Mises theorem regarding the posterior density, we adopt the techniques  from \cite{BorwankerKallianpurPrakasaRao1971} (see also \cite{PrakasaRao2000} and \cite{Bishwal2002}), where we slightly weaken one of the conditions compared to some previous versions of Bernstein-Von Mises theorem (see condition (C2) in Theorem \ref{thm:BvM}), which is also easier to verify.

Let $\Lambda:=\sqrt{I_{N}}(\Theta-\widehat{\theta}_{N})$, representing the normalized difference between the randomized parameter $\Theta$ and $\wh{\theta}_{N}$. By \eqref{eq:PosteriorThetaDeri}, the corresponding posterior density is then given by
\begin{align}\label{eq:PosteriorLambda}
\wt{p}(\lambda\,|\,U_{N})=p\bigg(\frac{\lambda}{\sqrt{I_{N}}}+\wh{\theta}_{N}\,\bigg|\,U_{N}\bigg)\frac{\dif\theta}{\dif\lambda}=C_{N}^{-1}\varrho\bigg(\frac{\lambda}{\sqrt{I_{N}}}+\wh{\theta}_{N}\bigg)\,e^{-\lambda^{2}/2},\quad\lambda\in(-\sqrt{I_{N}}\wh{\theta}_{N},\infty),
\end{align}
where
\begin{align}\label{eq:DefCN}
C_{N}:=\sqrt{I_{N}}\int_{0}^{\infty}\exp\bigg(-\frac{I_{N}}{2}\big(\eta-\wh{\theta}_{N}\big)^{2}\bigg)\varrho(\eta)\,\dif\eta=\int_{-\sqrt{I_{N}}\,\wh{\theta}_{N}}^{\infty}\varrho\bigg(\frac{\lambda}{\sqrt{I_{N}}}+\wh{\theta}_{N}\bigg)\,e^{-\lambda^{2}/2}\dif\lambda.
\end{align}
For notational convenience, we will extend the domains of $\varrho$ and $\wt{p}$ to $\bR$, with $\varrho(\theta)=0$ for $\theta\in(-\infty,0]$ and $\wt{p}(\lambda|U_{N})=0$ for $\lambda\in(-\infty,-\sqrt{I_{N}}\wh{\theta}_{N}]$. By \eqref{eq:PosteriorLambda}, the definition \eqref{eq:DefWTBetaNeta} of $\wt{\beta}_{N}$ can be written as
\begin{align}\label{eq:WTBetaNPosteriorLambda}
\wt{\beta}_{N}=\argmin_{\beta\in\bR_{+}}\int_{\bR}\ell\Big(\lambda+\sqrt{I_{N}}\big(\wh{\theta}_{N}-\beta\big)\Big)\,\wt{p}(\lambda\,|\,U_{N})\dif\lambda.
\end{align}

We are now in the position of presenting the Bernstein-von Mises theorem.
\begin{theorem}\label{thm:BvM}
Let $\varrho\in\cQ$ be positive and continuous in a neighborhood of $\theta_{0}$, and let $f$ be a non-negative, Borel-measurable function on $\bR$. Suppose that $\varrho$ and $f$ satisfy the following two conditions:
\begin{itemize}
\item [(C1)] there exists $\alpha\in(0,1)$ so that
    \begin{align*}
    \int_{\bR}f(x)\,e^{-\alpha x^{2}/2} \dif x<\infty;
    \end{align*}
\item [(C2)] for any $\delta>0$,
    \begin{align*}
    \lim_{N\rightarrow\infty}\int_{\{|\lambda|>\sqrt{I_{N}}\delta\}}f(\lambda)\,\varrho\bigg(\wh{\theta}_{N}+\frac{\lambda}{\sqrt{I_{N}}}\bigg)\,e^{-\lambda^{2}/2}\dif\lambda=0,\quad\bP-\text{a.s.}.
    \end{align*}
\end{itemize}
Then,
\begin{align}\label{eq:BvM}
\lim_{N\rightarrow\infty}\int_{\bR}f(\lambda)\bigg|\wt{p}(\lambda\,|\,U_{N})-\frac{e^{-\lambda^{2}/2}}{\sqrt{2\pi}}\bigg|\dif\lambda=0,\quad\bP-\text{a.s.}.
\end{align}
\end{theorem}

\begin{remark}\label{rem:BvM}
The above theorem implies that, for $\bP-$a.s.$\,\omega$, the posterior distribution measure of $\Theta$ converges weakly to the Dirac measure centered at $\theta_{0}$, as the number of Fourier modes increases. Indeed, let $g$ be any continuous and bounded function on $\bR$. Without loss of generality, assume that $g$ is non-negative (otherwise, consider $g^{+}$ and $g^{-}$ separately). By \eqref{eq:PosteriorLambda}, for any $\omega\in\Omega$,
\begin{align*}
\int_{\bR}g(\theta)\,p(\theta\,|\,U_{N})(\omega)\dif\theta=\int_{\bR}g\bigg(\frac{\lambda}{\sqrt{I_{N}}}+\wh{\theta}_{N}(\omega)\bigg)\,\wt{p}(\lambda\,|\,U_{N})(\omega)\dif\lambda.
\end{align*}
For each given $\omega\in\Omega$, we put $f_{\omega}(\lambda):=g(\lambda/\sqrt{I_{N}}+\wh{\theta}_{N}(\omega))$, $\lambda\in\bR$. In light of the boundedness of $g$, $f_{\omega}$ satisfies conditions (C1) and (C2) above, for any given $\omega$. From the proof of Theorem \ref{thm:BvM} (see \eqref{eq:fDiffTildep} below), equality \eqref{eq:BvM} holds for $f_{\omega}$ at any $\omega$. Together with the strong consistency of $\wh{\theta}_{N}$ (recalling \eqref{eq:MLEIniStrConst}), for $\bP-$a.s. $\omega\in\Omega$, we deduce that
\begin{align*}
\lim_{N\rightarrow\infty}\int_{\bR}g(\theta)\,p(\theta\,|\,U_{N})(\omega)\dif\theta&=\lim_{N\rightarrow\infty}\int_{\bR}g\bigg(\frac{\lambda}{\sqrt{I_{N}}}+\wh{\theta}_{N}(\omega)\bigg)\,\wt{p}(\lambda\,|\,U_{N})(\omega)\dif\lambda\\
&=\lim_{N\rightarrow\infty}\int_{\bR}f_{\omega}(\lambda)\,\wt{p}(\lambda\,|\,U_{N})(\omega)\dif\lambda\\
&=\lim_{N\rightarrow\infty}\frac{1}{\sqrt{2\pi}}\int_{\bR}f_{\omega}(\lambda)\,e^{-\lambda^{2}/2}\dif\lambda\\
&=\lim_{N\rightarrow\infty}\frac{1}{\sqrt{2\pi}}\int_{\bR}g\bigg(\frac{\lambda}{\sqrt{I_{N}}}+\wh{\theta}_{N}(\omega)\bigg)e^{-\lambda^{2}/2}\dif\lambda=g(\theta_{0}),
\end{align*}
where we used the dominated convergence theorem in the last equality.
\end{remark}

\begin{remark}\label{rem:ExpGrowPriorLossFunt}
Let $\varrho\in\cQ$, and let $f$ be a non-negative, Borel-measurable function on $\bR$. Assume that there exists $c_{1},c_{2}>0$ and $r\in(0,2)$ such that
\begin{align}\label{eq:ExpGrowPriorLossFunt}
\varrho(\theta)\leq c_{1}e^{c_{2}\theta^{r}},\quad\text{for all }\,\theta\in\bR_{+};\qquad f(x)\leq c_{1}e^{c_{2}|x|^{r}},\quad\text{for all }\,x\in\bR.
\end{align}
Clearly such $f$ satisfies condition (C1) in Theorem \ref{thm:BvM}. Moreover, $\varrho$ and $f$ also satisfy condition (C2) in Theorem \ref{thm:BvM}. Indeed, by the strong consistency of $\wh{\theta}_{N}$, for $\bP-$a.s. $\omega\in\Omega$ and $N\in\bN$ large enough (depending on $\omega$), we have that
\begin{align}
0&\leq\int_{\{|\lambda|>\sqrt{I_{N}}\delta\}}f(\lambda)\,\varrho\bigg(\wh{\theta}_{N}+\frac{\lambda}{\sqrt{I_{N}}}\bigg)\,e^{-\lambda^{2}/2}\dif\lambda\nonumber\\
&\leq c_{1}^{2}\int_{\{|\lambda|>\sqrt{I_{N}}\delta\}}\exp\bigg(c_{2}|\lambda|^{r}+c_{2}\bigg(\theta_{0}+1+\frac{|\lambda|}{\sqrt{I_{N}}}\bigg)^{r}-\frac{\lambda^{2}}{2}\bigg)\1_{\{\wh{\theta}_{N}+\lambda/\sqrt{I_{N}}>0\}}\dif\lambda\nonumber\\
\label{eq:ExpGrowPriorLosFuntC2} &\leq c_{1}^{2}\sup_{N\in\bN,\,\lambda\in\bR}\left(\exp\bigg(\bigg(\theta_{0}+1+\frac{|\lambda|}{\sqrt{I_N}}\bigg)^{r}-\frac{\lambda^{2}}{4}\bigg)\right)\cdot\int_{\{|\lambda|>\sqrt{I_{N}}\delta\}}\exp\bigg(c_{2}|\lambda|^{r}-\frac{\lambda^{2}}{4}\bigg)\dif\lambda\rightarrow 0,
\end{align}
as $N\rightarrow\infty$. In particular, all conditions in Theorem \ref{thm:BvM} are valid when $\varrho\in\cQ_{e,2}$ and $f\in\mathbf{W}_{e,2}$.
\end{remark}

The proof of Theorem \ref{thm:BvM} is based on the following technical lemma.

\begin{lemma}\label{lem:BvM}
Under the conditions of Theorem \ref{thm:BvM},
\begin{itemize}
\item [(a)] there exist $\delta_{0}>0$ such that
    \begin{align*}
    \lim_{N\rightarrow\infty}\int_{\{|\lambda|\leq\delta_{0}\sqrt{I_{N}}\}}f(\lambda)\bigg|\varrho\bigg(\frac{\lambda}{\sqrt{I_{N}}}+\wh{\theta}_{N}\bigg)-\varrho(\theta_{0})\bigg|\,e^{-\lambda^{2}/2}\dif\lambda=0,\quad\bP-\text{a.s.};
    \end{align*}
\item [(b)] for any $\delta>0$,
    \begin{align*}
    \lim_{N\rightarrow\infty}\int_{\{|\lambda|>\delta\sqrt{I_{N}}\}}f(\lambda)\bigg|\varrho\bigg(\frac{\lambda}{\sqrt{I_{N}}}+\wh{\theta}_{N}\bigg)-\varrho(\theta_{0})\bigg|\,e^{-\lambda^{2}/2}\dif\lambda=0,\quad\bP-\text{a.s.}.
	\end{align*}
\end{itemize}
\end{lemma}

\begin{proof}
Pick $\delta_{0}\in(0,\theta_{0}/2)$ such that $\varrho$ is continuous on $[\theta_{0}-2\delta_{0},\theta_{0}+2\delta_{0}]$. Recall that the Fisher information $I_{N}$ is unbounded, and thus, for any $C>0$, there exists $N\in\mathbb{N}$ so that $\delta_{0}\sqrt{I_{N}}>C$. We decompose the integral from part (a) as follows
\begin{align}\label{eq:DecompIntALemmaBvM}
\int_{\{|\lambda|\leq\delta_{0}\sqrt{I_{N}}\}}\!\!=\1_{\{|\wh{\theta}_{N}-\theta_{0}|>\delta_{0}\}}\!\int_{\{|\lambda|\leq\delta_{0}\sqrt{I_{N}}\}}\!\!+\1_{\{|\wh{\theta}_{N}-\theta_{0}|\leq\delta_{0}\}}\!\int_{\{|\lambda|\leq C\}}\!\!+\1_{\{|\wh{\theta}_{N}-\theta_{0}|\leq\delta_{0}\}}\!\int_{\{C<|\lambda|\leq\delta_{0}\sqrt{I_{N}}\}}.
\end{align}
By the strong consistency of $\wh{\theta}_{N}$, $|\wh{\theta}_{N}(\omega)-\theta_{0}|\leq\delta_{0}$, $\bP-$a.s., for sufficiently large $N$ (that may depend on $\omega$). Hence
\begin{align}\label{eq:LimitIntALemmaBvM1}
\lim_{N\rightarrow\infty}\1_{\{|\wh{\theta}_{N}-\theta_{0}|>\delta_{0}\}}\int_{\{|\lambda|\leq\delta_{0}\sqrt{I_{N}}\}}f(\lambda)\bigg|\varrho\bigg(\frac{\lambda}{\sqrt{I_{N}}}+\wh{\theta}_{N}\bigg)-\varrho(\theta_{0})\bigg|\,e^{-\lambda^{2}/2}\dif\lambda=0,\quad\bP-\text{a.s.}.
\end{align}
Moreover, by the strong consistency of $\wh{\theta}_{N}$ and the continuity of $\varrho$ on $[\theta_{0}-2\delta_{0},\theta_{0}+2\delta_{0}]$, for any $\lambda\in[-C,C]$,
\begin{align*}
\lim_{N\rightarrow\infty}\bigg|\varrho\bigg(\frac{\lambda}{\sqrt{I_{N}}}+\wh{\theta}_{N}\bigg)-\varrho(\theta_{0})\bigg|=0,\quad\bP-\text{a.s.}.
\end{align*}
Hence, by condition (C1) and the dominated convergence theorem,
\begin{align}\label{eq:LimitIntALemmaBvM2}
\lim_{N\rightarrow\infty}\1_{\{|\wh{\theta}_{N}-\theta_{0}|\leq\delta_{0}\}}\int_{\{|\lambda|\leq C\}}f(\lambda)\bigg|\varrho\bigg(\frac{\lambda}{\sqrt{I_{N}}}+\wh{\theta}_{N}\bigg)-\varrho(\theta_{0})\bigg|\,e^{-\lambda^{2}/2}\dif\lambda=0,\quad\bP-\text{a.s.}.
\end{align}
Finally, by  (C1) and the boundedness of $\varrho$ on $[\theta_{0}-2\delta_{0},\theta_{0}+2\delta_{0}]$,
\begin{align}
&\1_{\{|\widehat{\theta}_{N}-\theta_{0}|\leq\delta_{0}\}}\int_{C<|\lambda|\leq\delta_{0}\sqrt{I_{N}}}f(\lambda)\bigg|\varrho\bigg(\frac{\lambda}{\sqrt{I_{N}}}+\wh{\theta}_{N}\bigg)-\varrho(\theta_{0})\bigg|\,e^{-\lambda^{2}/2}\dif\lambda\nonumber\\
\label{eq:EstIntALemmaBvM3} &\quad\leq 2M\int_{\{|\lambda|>C\}}f(\lambda)\,e^{-\alpha\lambda^{2}/2}\,e^{-(1-\alpha)\lambda^{2}/2}\,\dif\lambda\leq 2M\,e^{-(1-\alpha)C^{2}/2}\int_{\bR}f(\lambda)\,e^{-\alpha\lambda^{2}/2}\,d\lambda,
\end{align}
where $M:=\sup_{\theta\in[\theta_{0}-2\delta_{0},\theta_{0}+2\delta_{0}]}\varrho(\theta)$. Combining \eqref{eq:DecompIntALemmaBvM}$-$\eqref{eq:EstIntALemmaBvM3}, and since $C>0$ is arbitrarily chosen, we complete the proof of part (a).

As for part (b), note that for any $\delta>0$,
\begin{align*}
&\int_{\{|\lambda|>\delta\sqrt{I_{N}}\}}f(\lambda)\bigg|\varrho\bigg(\frac{\lambda}{\sqrt{I_{N}}}+\wh{\theta}_{N}\bigg)-\varrho(\theta_{0})\bigg|\,e^{-\lambda^{2}/2}\dif\lambda\\
&\quad\leq\int_{\{|\lambda|>\delta\sqrt{I_{N}}\}}f(\lambda)\,\varrho\bigg(\frac{\lambda}{\sqrt{I_{N}}}+\wh{\theta}_{N}\bigg)\,e^{-\lambda^{2}/2}\dif\lambda+\varrho(\theta_{0})e^{-(1-\alpha)\delta^{2}I_{N}}\!\int_{\{|\lambda|>\delta\sqrt{I_{N}}\}}f(\lambda)\,e^{-\alpha\lambda^{2}/2}\dif\lambda,
\end{align*}
and in view of (C1) and (C2), (b) follows at once. The proof is now complete.
\end{proof}

\noindent
{\it Proof of Theorem \ref{thm:BvM}.} Since the constant function $f\equiv 1$ satisfies (C1) and (C2), by Lemma \ref{lem:BvM},
\begin{align*}
\lim_{N\rightarrow\infty}\int_{\bR}e^{-\lambda^{2}/2}\bigg|\varrho\bigg(\frac{\lambda}{\sqrt{I_{N}}}+\wh{\theta}_{N}\bigg)-\varrho(\theta_{0})\bigg|\dif\lambda=0,\quad\bP-\text{a.s.},
\end{align*}
which, together with \eqref{eq:DefCN}, implies that
\begin{align*}
\lim_{N\rightarrow\infty}C_{N}:=\lim_{N\rightarrow\infty}\int_{\bR}\varrho\bigg(\frac{\lambda}{\sqrt{I_{N}}}+\wh{\theta}_{N}\bigg)\,e^{-\lambda^{2}/2}\dif\lambda=\varrho(\theta_{0})\int_{\bR}e^{-\lambda^{2}/2}\dif\lambda=\sqrt{2\pi}\varrho(\theta_{0}),\quad\bP-\text{a.s.}.
\end{align*}
Therefore, by \eqref{eq:PosteriorLambda}, as $N\rightarrow\infty$,
\begin{align}
\int_{\bR}f(\lambda)\bigg|\wt{p}(\lambda\,|\,U_{N})-\frac{e^{-\lambda^{2}/2}}{\sqrt{2\pi}}\bigg|\dif\lambda&\leq C_{N}^{-1}\int_{\bR}f(\lambda)\bigg|\varrho\bigg(\frac{\lambda}{\sqrt{I_{N}}}+\wh{\theta}_{N}\bigg)-\varrho(\theta_{0})\bigg|\,e^{-\lambda^{2}/2}\dif\lambda\nonumber\\
\label{eq:fDiffTildep} &\quad +\bigg|C^{-1}_{N}\varrho(\theta_{0})-\frac{1}{\sqrt{2\pi}}\bigg|\int_{\bR}f(\lambda)\,e^{-\lambda^{2}/2}\dif\lambda\rightarrow 0,\quad\bP-\text{a.s.},
\end{align}
which completes the proof of the theorem.\hfill $\Box$

\medskip
In order to investigate the uniform asymptotic properties of $\wt{\beta}_{N}$, we need the following uniform version of Bernstein-Von Mises Theorem.

\begin{theorem}\label{thm:BvMUnifConvProb}
Let $\varrho\in\cQ$ be continuous and positive on $\bR_{+}$, and let $f$ be a non-negative, Borel-measurable function on $\bR$. Assume that $\varrho$ and $f$ satisfy \eqref{eq:ExpGrowPriorLossFunt}. Then, for any compact set $B\subset\bR_{+}$ and any $\varepsilon>0$,
\begin{align*}
\lim_{N\rightarrow\infty}\sup_{\theta_{0}\in B}\bP\bigg(\int_{\bR}f(\lambda)\bigg|\wt{p}(\lambda\,|\,U_{N})-\frac{e^{-\lambda^{2}/2}}{\sqrt{2\pi}}\bigg|\dif\lambda>\varepsilon\bigg)=0.
\end{align*}
\end{theorem}

The following technical lemma is a key ingredient for the proof of Theorem \ref{thm:BvMUnifConvProb}.

\begin{lemma}\label{lem:BvMUnifConvProb}
Under the conditions of Theorem \ref{thm:BvMUnifConvProb}, for any compact set $B\subset\bR_{+}$,
\begin{itemize}
\item [(a)] there exists $\delta_{0}>0$, such that for any $\varepsilon>0$,
    \begin{align*}
    \lim_{N\rightarrow\infty}\sup_{\theta_{0}\in B}\bP\left(\int_{\{|\lambda|\leq\delta_{0}\sqrt{I_{N}}\}}f(\lambda)\bigg|\varrho
    \bigg(\frac{\lambda}{\sqrt{I_{N}}}+\wh{\theta}_{N}\bigg)-\varrho(\theta_{0})\bigg|\,e^{-\lambda^{2}/2}\dif\lambda>\varepsilon\right)=0;
    \end{align*}
\item [(b)] for any $\delta>0$ and $\varepsilon>0$,
    \begin{align*}
    \lim_{N\rightarrow\infty}\sup_{\theta_{0}\in B}\bP\left(\int_{\{|\lambda|>\delta\sqrt{I_{N}}\}}f(\lambda)\bigg|\varrho\bigg(\frac{\lambda}{\sqrt{I_{N}}}+\wh{\theta}_{N}\bigg)
    -\varrho(\theta_{0})\bigg|\,e^{-\lambda^{2}/2}\dif\lambda>\varepsilon\right)=0.
    \end{align*}
\end{itemize}
\end{lemma}

\begin{proof}
The proof follows the same line as that of Lemma \ref{lem:BvM}. Pick $\delta_{0}\in(0,\inf_{\theta\in B}(\theta/2))$. Since $I_{N}\to+\infty$, for any $C>0$, there exists $N\in\mathbb{N}$ so that $\delta_{0}\sqrt{I_{N}}>C$. We decompose the integral from part (a) as in \eqref{eq:DecompIntALemmaBvM}. By the uniform weak consistency of $\wh{\theta}_{N}$ (property (a) in Corollary \ref{cor:MLEIniUAN}), we have that
\begin{align}
&\lim_{N\rightarrow\infty}\sup_{\theta_{0}\in B}\bP\left(\1_{\{|\wh{\theta}_{N}-\theta_{0}|>\delta_{0}\}}\int_{\{|\lambda|\leq\delta_{0}\sqrt{I_{N}}\}}f(\lambda)\bigg|\varrho\bigg(\frac{\lambda}{\sqrt{I_{N}}}+\wh{\theta}_{N}\bigg)-\varrho(\theta_{0})\bigg|\,e^{-\lambda^{2}/2}\dif\lambda>\frac{\varepsilon}{3}\right)\nonumber\\
\label{eq:BvMUnifConvProbA1} &\quad\leq\lim_{N\rightarrow\infty}\sup_{\theta_{0}\in B}\bP\Big(\big|\wh{\theta}_{N}-\theta_{0}\big|>\delta_{0}\Big)=0.
\end{align}
Moreover, for any $\gamma\in(0,\inf_{\theta\in B}\theta)$, let $M_{\gamma}:=\sup_{\{\theta\in B_{\gamma}\}}\varrho(\theta)$, where $B_{\gamma}:=\{\theta\in\bR_{+}:\,\inf_{\eta\in B}|\eta-\theta|\leq\gamma\}$ is a compact subset of $\bR_{+}$. Since $C>0$ is arbitrary, an argument similar to those leading to \eqref{eq:EstIntALemmaBvM3} implies that
\begin{align}
&\sup_{\theta_{0}\in B}\bP\left(\1_{\{|\wh{\theta}_{N}-\theta|\leq\delta_{0}\}}\int_{\{C<|\lambda|\leq\delta_{0}\sqrt{I_{N}}\}}f(\lambda)\bigg|\varrho\bigg(\frac{\lambda}{\sqrt{I_{N}}}+\wh{\theta}_{N}\bigg)-\varrho(\theta)\bigg|\,e^{-\lambda^{2}/2}\dif\lambda>\frac{\varepsilon}{3}\right)\nonumber\\
\label{eq:BvMUnifConvProbA2} &\quad\leq\sup_{\theta_{0}\in B}\bP\bigg(2M_{2\delta_{0}}\,e^{-(1-\alpha)C^{2}/2}\int_{\bR}f(\lambda)\,e^{-\alpha\lambda^{2}/2}\,d\lambda>\frac{\varepsilon}{3}\bigg)=0.
\end{align}
Finally, since $\varrho$ is uniformly continuous on $B_{\delta_{0}}$, there exists $\delta_{1}\in(0,\delta_{0})$ such that, for any $\theta,\eta\in B_{\delta_{0}}$ with $|\theta-\eta|\leq\delta_{1}$, $|\varrho(\theta)-\varrho(\eta)|\leq\varepsilon/(6CK)$, where $K:=\sup_{|\lambda|\leq C}f(\lambda)<\infty$ in view of \eqref{eq:ExpGrowPriorLossFunt}. By the uniform weak consistency of $\wh{\theta}_{N}$ (property (a) in Corollary \ref{cor:MLEIniUAN}), for $N\in\bN$ large enough (so that $C/\sqrt{I_{N}}\leq\delta_{1}/2$), we deduce that
\begin{align}
&\sup_{\theta_{0}\in B}\bP\left(\1_{\{|\wh{\theta}_{N}-\theta_{0}|\leq\delta_{0}\}}\int_{\{|\lambda|\leq C\}}f(\lambda)\bigg|\varrho\bigg(\frac{\lambda}{\sqrt{I_{N}}}+\wh{\theta}_{N}\bigg)-\varrho(\theta_{0})\bigg|\,e^{-\lambda^{2}/2}\dif\lambda>\frac{\varepsilon}{3}\right)\nonumber\\
&\quad\leq\sup_{\theta_{0}\in B}\bP\left(\sup_{|\lambda|\leq C}\bigg|\varrho\bigg(\frac{\lambda}{\sqrt{I_{N}}}+\wh{\theta}_{N}\bigg)-\varrho(\theta_{0})\bigg|>\frac{\varepsilon}{6CK},\,\big|\wh{\theta}_{N}-\theta_{0}\big|\leq\frac{\delta_{1}}{2}\right)+\sup_{\theta_{0}\in B}\bP\bigg(\big|\wh{\theta}_{N}-\theta_{0}\big|>\frac{\delta_{1}}{2}\bigg)\nonumber\\
\label{eq:BvMUnifConvProbA3} &\quad =\sup_{\theta_{0}\in B}\bP\bigg(\big|\wh{\theta}_{N}-\theta_{0}\big|>\frac{\delta_{1}}{2}\bigg)\rightarrow 0,\quad N\rightarrow\infty.
\end{align}
Combining \eqref{eq:BvMUnifConvProbA1}$-$\eqref{eq:BvMUnifConvProbA3} completes the proof of part (a).

As for part (b), using similar arguments to those used in establishing \eqref{eq:ExpGrowPriorLosFuntC2}, for any $\delta>0$ and $\delta_{1}\in(0,\inf_{\theta\in B}\theta)$, we have that, as $N\rightarrow\infty$,
\begin{align*}
&\sup_{\theta\in B_{\delta_{1}}}\int_{\{|\lambda|>\delta\sqrt{I_{N}}\}}f(\lambda)\,\varrho\bigg(\frac{\lambda}{\sqrt{I_{N}}}+\theta\bigg)\,e^{-\lambda^{2}/2}\dif\lambda\\
&\quad\leq c_{1}^{2}\sup_{\theta\in B_{\delta_{1}},\lambda\in\bR,N\in\bN}\left(\exp\bigg(\bigg(\theta+\frac{\lambda}{\sqrt{I_N}}\bigg)^{r}-\frac{\lambda^{2}}{4}\bigg)\right)\cdot\int_{\{|\lambda|>\sqrt{I_{N}}\delta\}}\exp\bigg(c_{2}|\lambda|^{r}-\frac{\lambda^{2}}{4}\bigg)\dif\lambda\rightarrow 0.
\end{align*}
This, together with the uniform weak consistency of $\wh{\theta}_{N}$ (property (a) in Corollary~\ref{cor:MLEIniUAN}), implies that, for any $\varepsilon>0$ and $\delta>0$,
\begin{align*}
&\sup_{\theta_{0}\in B}\bP\left(\int_{\{|\lambda|>\delta\sqrt{I_{N}}\}}f(\lambda)\bigg|\varrho\bigg(\frac{\lambda}{\sqrt{I_{N}}}+\wh{\theta}_{N}\bigg)-\varrho(\theta_{0})\bigg|\,e^{-\lambda^{2}/2}\dif\lambda>\varepsilon\right)\\
&\quad\leq\sup_{\theta_{0}\in B}\bP\left(\int_{\{|\lambda|>\delta\sqrt{I_{N}}\}}f(\lambda)\,\varrho\bigg(\frac{\lambda}{\sqrt{I_{N}}}+\wh{\theta}_{N}\bigg)\,e^{-\lambda^{2}/2}\dif\lambda>\frac{\varepsilon}{2}\right)\\
&\qquad +\bP\left(\sup_{\theta_{0}\in B}\varrho(\theta_{0})\cdot\int_{\{|\lambda|>\delta\sqrt{I_{N}}\}}f(\lambda)\,e^{-\lambda^{2}/2}\dif\lambda>\frac{\varepsilon}{2}\right)\\
&\quad\leq\sup_{\theta_{0}\in B}\bP\left(\int_{\{|\lambda|>\delta\sqrt{I_{N}}\}}f(\lambda)\,\varrho\bigg(\frac{\lambda}{\sqrt{I_{N}}}+\wh{\theta}_{N}\bigg)\,e^{-\lambda^{2}/2}\dif\lambda>\frac{\varepsilon}{2},\,\wh{\theta}_{N}\in B_{\delta_{1}}\right)+\sup_{\theta_{0}\in B}\bP\big(\wh{\theta}_{N}\notin B_{\delta_{1}}\big)\\
&\quad\leq\bP\left(\sup_{\theta\in B_{\delta_{1}}}\int_{\{|\lambda|>\delta\sqrt{I_{N}}\}}f(\lambda)\,\varrho\bigg(\frac{\lambda}{\sqrt{I_{N}}}+\theta\bigg)\,e^{-\lambda^{2}/2}\dif\lambda>\frac{\varepsilon}{2}\right)+\sup_{\theta_{0}\in B}\bP\Big(\big|\wh{\theta}_{N}-\theta_{0}\big|>\delta_{1}\Big)\rightarrow 0,
\end{align*}
as $N\rightarrow\infty$, which completes the proof of part (b).
\end{proof}

\noindent
\textit{Proof of Theorem \ref{thm:BvMUnifConvProb}.} Since $B\subset\bR_{+}$ is compact and $\rho$ is positive on $\bR_{+}$, $L:=\inf_{\theta\in B}\varrho(\theta)>0$. For any $\varepsilon>0$, by \eqref{eq:PosteriorLambda} and \eqref{eq:DefCN}, we have that
\begin{align}
&\sup_{\theta_{0}\in B}\bP\left(\int_{\bR}f(\lambda)\,\bigg|\wt{p}(\lambda\,|\,U_{N})-\frac{e^{-\lambda^{2}/2}}{\sqrt{2\pi}}\bigg|\dif\lambda>\varepsilon\right)\nonumber\\
&\quad\leq\sup_{\theta_{0}\in B}\bP\left(C_{N}^{-1}\int_{\bR}f(\lambda)\bigg|\varrho\bigg(\frac{\lambda}{\sqrt{I_{N}}}+\wh{\theta}_{N}\bigg)-\varrho(\theta_{0})\bigg|\,e^{-\lambda^{2}/2}\dif\lambda>\frac{\varepsilon}{2}\right)\nonumber\\
&\qquad +\sup_{\theta_{0}\in B}\bP\left(\bigg|C^{-1}_{N}\varrho(\theta_{0})-\frac{1}{\sqrt{2\pi}}\bigg|\int_{\bR}f(\lambda)\,e^{-\lambda^{2}/2}\dif\lambda>\frac{\varepsilon}{2}\right)\nonumber\\
&\quad\leq\sup_{\theta_{0}\in B}\bP\left(\int_{\bR}f(\lambda)\bigg|\varrho\bigg(\frac{\lambda}{\sqrt{I_{N}}}+\wh{\theta}_{N}\bigg)-\varrho(\theta_{0})\bigg|\,e^{-\lambda^{2}/2}\dif\lambda>\frac{\varepsilon}{2}C_{N},\,C_{N}\geq\frac{L}{2}\right)\nonumber\\
&\qquad +2\sup_{\theta_{0}\in B}\bigg(C_{N}<\frac{L}{2}\bigg)+\sup_{\theta_{0}\in B}\bP\left(\big|\sqrt{2\pi}\varrho(\theta_{0})-C_{N}\big|\int_{\bR}f(\lambda)\,e^{-\lambda^{2}/2}\dif\lambda>\frac{\sqrt{2\pi}\varepsilon}{2}C_{N},\,C_{N}\geq\frac{L}{2}\right)\nonumber\\
&\quad\leq\sup_{\theta_{0}\in B}\bP\left(\int_{\bR}f(\lambda)\bigg|\varrho\bigg(\frac{\lambda}{\sqrt{I_{N}}}+\wh{\theta}_{N}\bigg)-\varrho(\theta_{0})\bigg|\,e^{-\lambda^{2}/2}\dif\lambda>\frac{L\varepsilon}{4}\right)+2\sup_{\theta_{0}\in B}\bigg(\big|C_{N}-\sqrt{2\pi}\varrho(\theta_{0})\big|>\frac{L}{2}\bigg)\nonumber\\
\label{eq:DecompBvMUnifConvProb} &\qquad +\sup_{\theta_{0}\in B}\bP\left(\big|\sqrt{2\pi}\varrho(\theta_{0})-C_{N}\big|\int_{\bR}f(\lambda)\,e^{-\lambda^{2}/2}\dif\lambda>\frac{\sqrt{2\pi}L\varepsilon}{4}\right).
\end{align}
By Lemma \ref{lem:BvMUnifConvProb}, clearly we have
\begin{align}\label{eq:LimitBvMUnifConvProb1}
\lim_{N\rightarrow\infty}\sup_{\theta_{0}\in B}\bP\left(\int_{\bR}f(\lambda)\bigg|\varrho\bigg(\frac{\lambda}{\sqrt{I_{N}}}+\wh{\theta}_{N}\bigg)-\varrho(\theta_{0})\bigg|\,e^{-\lambda^{2}/2}\dif\lambda>\frac{L\varepsilon}{4}\right)=0.
\end{align}
Moreover, for any $\delta>0$, by \eqref{eq:DefCN} and Lemma \ref{lem:BvMUnifConvProb} (with $f\equiv 1$) we have, as $N\rightarrow\infty$, that
\begin{align}\label{eq:LimitBvMUnifConvProb2}
\sup_{\theta_{0}\in B}\bP\Big(\big|C_{N}-\sqrt{2\pi}\varrho(\theta_{0})\big|>\delta\Big)\leq\sup_{\theta_{0}\in B}\bP\left(\int_{\bR}\bigg|\varrho\bigg(\frac{\lambda}{\sqrt{I_{N}}}+\wh{\theta}_{N}\bigg)-\varrho(\theta_{0})\bigg|\,e^{-\lambda^{2}/2}\dif\lambda>\delta\right)\rightarrow 0.
\end{align}
Combining \eqref{eq:DecompBvMUnifConvProb}$-$\eqref{eq:LimitBvMUnifConvProb2} completes the proof of the theorem.\hfill$\Box$

\subsection{Asymptotic Properties of $\wt{\beta}_{N}$}\label{sec:AsymTildeBetaN}

With the help of Theorems \ref{thm:BvM} and \ref{thm:BvMUnifConvProb}, we will now study the asymptotic properties of $\wt{\beta}_N$ with respect to the set of loss functions $\ell$ and prior densities $\varrho$ which have exponential growth rates.

\begin{theorem}\label{thm:BayeUANExp}
Let $\varrho\in\cQ_{e,2}$ and $\ell\in\mathbf{W}_{e,2}$. Assume that $r\mapsto\int_{\bR}{\ell}(\lambda+r)e^{-\lambda^{2}/2}\dif\lambda$ has a strict minimum at $r=0$, and that $\wt{\beta}_{N}$ is well defined with respect to each loss function $\ell\in\mathbf{W}_{e,2}$, for every $N\in\bN$. Then,
\begin{align}\label{eq:ConvasWTBetaN}
\lim_{N\rightarrow\infty}\sqrt{I_{N}}\big(\wt{\beta}_{N}-\wh{\theta}_{N}\big)=0,\quad\bP-\text{a.s.}.
\end{align}
Moreover, for any compact set $B\subset\bR_{+}$ and any $\varepsilon>0$,
\begin{align}\label{eq:UnifConvProbWTBetaN}
\lim_{N\rightarrow\infty}\sup_{\theta_{0}\in B}\bP\Big(\sqrt{I_{N}}\big|\wt{\beta}_{N}-\wh{\theta}_{N}\big|>\varepsilon\Big)=0.
\end{align}
\end{theorem}

\begin{proof}
For any $r\in\bR$, denote by
\begin{align*}
\psi(r):=\int_{\bR}{\ell}(\lambda+r)\,\frac{e^{-\lambda^{2}/2}}{\sqrt{2\pi}}\dif\lambda,
\end{align*}
which has a strict minimum at $r=0$. Note that the integral above is finite since $\ell\in\mathbf{W}_{e,2}$. For any $\varepsilon>0$ and $N\in\bN$, we first have that
\begin{align}
&\bP\Big(\sqrt{I_{N}}\big|\wt{\beta}_{N}-\wh{\theta}_{N}\big|>\varepsilon\Big)\nonumber\\
&\quad\leq\bP\bigg(\inf_{\beta:\sqrt{I_{N}}|\beta-\wh{\theta}_{N}|>\varepsilon}\int_{\bR_{+}}\ell\big(\sqrt{I_{N}}(\eta-\beta)\big)\,p(\eta\,|\,U_{N})\dif\eta<\int_{\bR_{+}}\ell\Big(\sqrt{I_{N}}\big(\eta-\wh{\theta}_{N}\big)\Big)p(\eta\,|\,U_{N})\dif\eta\bigg)\nonumber\\
&\quad =\bP\bigg(\inf_{|r|>\varepsilon}\int_{\bR}\ell(\lambda+r)\,\wt{p}(\lambda\,|\,U_{N})\dif\lambda<\int_{\bR}\ell(\lambda)\,\wt{p}(\lambda\,|\,U_{N})\dif\lambda\bigg)\nonumber\\
\label{eq:UpperBoundDistWTBetaWHTheta1} &\quad\leq\bP\left(\inf_{|r|>\varepsilon}\int_{\bR}\ell(\lambda+r)\,\wt{p}(\lambda\,|\,U_{N})\dif\lambda<\psi(0)+\int_{\bR}\ell(\lambda)\bigg|\wt{p}(\lambda\,|\,U_{N})-\frac{e^{-\lambda^{2}/2}}{\sqrt{2\pi}}\bigg|\dif\lambda\right).
\end{align}
Pick $\delta=\delta(\varepsilon)\in(0,\inf_{|r|>\varepsilon}\psi(r)-\psi(0))$. Since $\ell$ is symmetric on $\bR$ and is non-decreasing on $\bR_{+}$, there exists $m=m(\varepsilon,\delta)>\varepsilon$ and $M=M(\varepsilon,\delta)>m$ such that
\begin{align}\label{eq:ChoiceMm}
\int_{-m}^{m}\frac{e^{-\lambda^{2}/2}}{\sqrt{2\pi}}\dif\lambda\geq\frac{\inf_{|r|>\varepsilon}\psi(r)-\delta/2}{\inf_{|r|>\varepsilon}\psi(r)-\delta/4},\quad\inf_{|x|>M-m}\ell(x)\geq\inf_{|r|>\varepsilon}\psi(r)-\frac{\delta}{4}.
\end{align}
To obtain the existence of $M$ above, assume that for any $x\in\bR$, $\ell(x)\leq\inf_{|r|>\varepsilon}\psi(r)-\delta/4$, then
\begin{align*}
\inf_{|r|>\varepsilon}\psi(r)=\inf_{|r|>\varepsilon}\int_{\bR}{\ell}(\lambda+r)\,\frac{e^{-\lambda^{2}/2}}{\sqrt{2\pi}}\dif\lambda\leq\inf_{|r|>\varepsilon}\psi(r)-\frac{\delta}{4},
\end{align*}
which is clearly a contradiction. Next, there exists $K=K(\varepsilon,\delta)>M$ large enough such that, whenever $|r|\in(\varepsilon,M)$,
\begin{align}\label{eq:ChoiceK}
\int_{-K}^{K}\ell(\lambda+r)\,\frac{e^{-\lambda^{2}/2}}{\sqrt{2\pi}}\dif\lambda\geq\int_{\bR}\ell(\lambda+r)\,\frac{e^{-\lambda^{2}/2}}{\sqrt{2\pi}}\dif\lambda-\delta\geq\inf_{|r|>\varepsilon}\psi(r)-\frac{\delta}{2}.
\end{align}
Note that for $|r|\geq M$, \eqref{eq:ChoiceMm} implies that
\begin{align}\label{eq:IntmpmBigr}
\int_{-m}^{m}\ell(\lambda+r)\,\frac{e^{-\lambda^{2}/2}}{\sqrt{2\pi}}\dif\lambda\geq\frac{\inf_{|r|>\varepsilon}\psi(r)-\delta/2}{\inf_{|r|>\varepsilon}\psi(r)-\delta/4}\cdot\bigg(\inf_{|r|>\varepsilon}\psi(r)-\frac{\delta}{4}\bigg)=\inf_{|r|>\varepsilon}\psi(r)-\frac{\delta}{2}.
\end{align}
Hence, by combining \eqref{eq:ChoiceK} and \eqref{eq:IntmpmBigr}, we obtain that
\begin{align}\label{eq:LowerBoundIntKellnormal}
\inf_{|r|>\varepsilon}\int_{-K}^{K}\ell(\lambda+r)\,\frac{e^{-\lambda^{2}/2}}{\sqrt{2\pi}}\dif\lambda\geq\inf_{|r|\geq\varepsilon}\psi(r)-\frac{\delta}{2}>\psi(0)+\frac{\delta}{2}.
\end{align}
Moreover, since $\ell$ is symmetric on $\bR$ and is non-decreasing on $\bR_{+}$, for any $|r|>\varepsilon+1+2K$,
\begin{align*}
\int_{-K}^{K}\ell(\lambda+r)\,\wt{p}(\lambda\,|\,U_{N})\dif\lambda\geq\sup_{|x|\leq\varepsilon+1+K}\!\!\ell(x)\cdot\int_{-K}^{K}\wt{p}(\lambda\,|\,U_{N})\dif\lambda\geq\int_{-K}^{K}\ell(\lambda+\varepsilon+1)\,\wt{p}(\lambda\,|\,U_{N})\dif\lambda.
\end{align*}
Hence, we have that
\begin{align*}
\inf_{|r|>\varepsilon}\!\int_{\bR}\ell(\lambda+r)\wt{p}(\lambda\,|\,U_{N})\!\dif\lambda\geq\inf_{|r|>\varepsilon}\!\int_{-K}^{K}\!\ell(\lambda+r)\wt{p}(\lambda\,|\,U_{N})\!\dif\lambda=\inf_{|r|\in(\varepsilon,\varepsilon+1+2K)}\!\int_{-K}^{K}\!\ell(\lambda+r)\wt{p}(\lambda\,|\,U_{N})\!\dif\lambda.
\end{align*}
Together with \eqref{eq:LowerBoundIntKellnormal}, we deduce that
\begin{align}
&\inf_{|r|>\varepsilon}\int_{\bR}\ell(\lambda+r)\,\wt{p}(\lambda\,|\,U_{N})\dif\lambda\nonumber\\
&\quad\geq\inf_{|r|\in(\varepsilon,\varepsilon+1+2K)}\int_{-K}^{K}\ell(\lambda+r)\,\frac{e^{-\lambda^{2}/2}}{\sqrt{2\pi}}\dif\lambda-\sup_{|r|\in(\varepsilon,\varepsilon+1+2K)}\int_{-K}^{K}\ell(\lambda+r)\bigg|\wt{p}(\lambda\,|\,U_{N})-\frac{e^{-\lambda^{2}/2}}{\sqrt{2\pi}}\bigg|\dif\lambda\nonumber\\
\label{eq:LowerBoundIntRellTildep} &\quad >\psi(0)+\frac{\delta}{2}-\sup_{|x|\leq\varepsilon+1+3K}\ell(x)\cdot\int_{\bR}\bigg|\wt{p}(\lambda\,|\,U_{N})-\frac{e^{-\lambda^{2}/2}}{\sqrt{2\pi}}\bigg|\dif\lambda.
\end{align}
Finally, by combining \eqref{eq:UpperBoundDistWTBetaWHTheta1} and \eqref{eq:LowerBoundIntRellTildep}, we obtain that for any $N\in\bN$,
\begin{align*}
\bP\Big(\sqrt{I_{N}}\big|\wt{\beta}_{N}-\wh{\theta}_{N}\big|>\varepsilon\Big)\leq\bP\left(\int_{\bR}\bigg(\ell(\lambda)+\sup_{|x|\leq\varepsilon+1+3K}\ell(x)\bigg)\bigg|\wt{p}(\lambda\,|\,U_{N})-\frac{e^{-\lambda^{2}/2}}{\sqrt{2\pi}}\bigg|\dif\lambda>\frac{\delta}{2}\right),
\end{align*}
and thus
\begin{align*}
\bP\Big(\limsup_{N\rightarrow\infty}\Big\{\sqrt{I_{N}}\big|\wt{\beta}_{N}-\wh{\theta}_{N}\big|>\varepsilon\Big\}\Big)\leq\bP\left(\limsup_{N\rightarrow\infty}\bigg\{\int_{\bR}\big(\ell(\lambda)+C\big)\bigg|\wt{p}(\lambda\,|\,U_{N})-\frac{e^{-\lambda^{2}/2}}{\sqrt{2\pi}}\bigg|\dif\lambda>\frac{\delta}{2}\bigg\}\right),
\end{align*}
where $C=C(\varepsilon,\delta,K):=\sup_{|x|\leq\varepsilon+1+3K}\ell(x)$. The results of the theorem follow immediately from Theorem~\ref{thm:BvM} and \ref{thm:BvMUnifConvProb}.
\end{proof}

\begin{remark}\label{rem:AsymetricLossFunt}
Theorem \ref{thm:BayeUANExp} remains valid when the symmetry of $\ell$ in condition (i) of Definition \ref{def:LossFunt} is relaxed to the non-increasing monotonicity on $(-\infty,0)$. In this case, one has to split the integral $\int_{\bR}\ell(\lambda+r)\,\wt{p}(\lambda\,|\,U_{N})\dif\lambda$ into the integrals on $[0,\infty)$ and $(-\infty,0)$, and to perform similar estimates for each integral. We skip the technical details here for the sake of brevity.
\end{remark}

\begin{proposition}\label{cor:BayeUANExp}
Under the conditions of Theorem~\ref{thm:BayeUANExp}, the asymptotic properties (a)$-$(c) stated in Theorem~\ref{thm:BayeUANPoly} remain true for $\wt{\beta}_{N}$. Moreover, $\wt{\beta}_{N}$ is strongly consistent. If in addition,  $\varrho(x)=\1_{\bR_{+}}(x)$,  is the uninformed/uniform prior, then  $\wt{\beta}_{N}$ also satisfies the asymptotic property $(d)$ in Theorem~\ref{thm:BayeUANPoly}.
\end{proposition}

\begin{proof}
Properties (a)$-$(c) in Theorem~\ref{thm:BayeUANPoly} follow immediately from the properties (a)$-$(c) in Proposition~\ref{cor:MLEIniUAN} for $\wh{\theta}_{N}$ and Theorem~\ref{thm:BayeUANExp}. The strong consistency is a direct consequence of \eqref{eq:MLEIniStrConst} and \eqref{eq:ConvasWTBetaN}.

Next, with $\ell\in\mathbf{W}_{e,2}$  and satisfying the conditions of Theorem \ref{thm:BayeUANExp}, we take $\varrho(x)=\1_{\bR_{+}}(x)$.
Due to \eqref{eq:PosteriorLambda} and \eqref{eq:WTBetaNPosteriorLambda}, we obtain
\begin{align}\label{eq:WTBetaNUnifPrior}
\wt{\beta}_{N}=\argmin_{\beta\in\bR_{+}}\int_{-\sqrt{I_{N}}\wh{\theta}_{N}}^{\infty}\ell\Big(\lambda+\sqrt{I_{N}}\big(\wh{\theta}_{N}-\beta\big)\Big)e^{-\lambda^{2}/2}\dif\lambda,
\end{align}
We will investigate the asymptotic efficiency of $\wt{\beta}_{N}$ as stated in Theorem \ref{thm:BayeUANPoly}.(d).  We will exclude any bounded loss function $\ell$ which clearly belong to $\mathbf{W}_{p}$, and thus their asymptotic efficiency is covered by Theorem~\ref{thm:BayeUANPoly}.

In view of the asymptotic normality of $\wt{\beta}_{N}$ with parameter $(\theta_{0},I_{N}^{-1})$ (property (b) above), uniformly in $\theta_{0}\in B$, and by \cite[Theorem III.1.3]{IbragimovKhasminskiiBook1981}, it is sufficient to show that for any $m\in\bN$, there exists $N_{0}=N_{0}(m)\in\bN$, such that the family of random variables $\{|\sqrt{I_{N}}(\wt{\beta}_{N}-\theta_{0})|^{m},N\geq N_{0},\theta_{0}\in B\}$ are uniformly integrable under $\bP$, for any given compact set $B\subset\bR_{+}$. In what follows, we will show that for any $m\in\bN$ and $\theta_{0}\in B$, there exists $N_{0}=N_{0}(m)\in\bN$ and $K=K(m)>0$, such that for any $N\geq N_{0}$ and $\omega\in\Omega$,
\begin{align}\label{eq:UpperBoundWTBetaNMLEIni}
\Big|\sqrt{I_{N}}\big(\wt{\beta}_{N}(\omega)-\theta_{0}\big)\Big|^{m}\leq\Big(\Big|\sqrt{I_{N}}\big(\wh{\theta}_{N}(\omega)-\theta_{0}\big)\Big|^{m}+K\Big),
\end{align}
which clearly implies the uniform integrability in light of \eqref{eq:MLEIniAsymNormal}.

Since $\ell$ is symmetric on $\bR$ and non-decreasing on $\bR_{+}$, there exists $x_{1}>0$ such that $\ell(x_{1})>0$, $\ell(x)\geq\ell(x-2x_{1})$ for all $x\geq x_{1}$, and $\ell(x)<\ell(x-2x_{1})$ for all $x\in[0,x_{1})$. For such $x_{1}$, we can find $x_{2}\in(0,x_{1})$ so that $\ell(x_{2})<\ell(x_{1})$. Then, for any $N\in\bN$ and $\beta\geq 2x_{1}/\sqrt{I_{N}}$,
\begin{align*}
&\int_{-\sqrt{I_{N}}\wh{\theta}_{N}}^{\infty}\bigg(\ell\Big(\lambda+\sqrt{I_{N}}\big(\wh{\theta}_{N}-\beta\big)\Big)-\ell\big(\lambda+\sqrt{I_{N}}\wh{\theta}_{N}\big)\bigg)\,e^{-\lambda^{2}/2}\dif\lambda\\
&\quad =\int_{\bR_{+}}\Big(\ell\big(y-\sqrt{I_{N}}\beta\big)-\ell(y)\Big)e^{-(y-\sqrt{I_{N}}\wh{\theta}_{N})^{2}/2}\dif y\\
&\quad =\int_{0}^{x_{1}}\!\Big(\ell\big(y-\sqrt{I_{N}}\beta\big)-\ell(y)\Big)e^{-(y-\sqrt{I_{N}}\wh{\theta}_{N})^{2}/2}\dif y+\!\int_{x_{1}}^{\infty}\!\Big(\ell\big(y-\sqrt{I_{N}}\beta\big)-\ell(y)\Big)e^{-(y-\sqrt{I_{N}}\wh{\theta}_{N})^{2}/2}\dif y\\
&\quad\geq\int_{0}^{x_{2}}\Big(\ell\big(y-\sqrt{I_{N}}\beta\big)-\ell(y)\Big)\,e^{-(y-\sqrt{I_{N}}\wh{\theta}_{N})^{2}/2}\dif y-\int_{x_{1}}^{\infty}\ell(y)\,e^{-(y-\sqrt{I_{N}}\wh{\theta}_{N})^{2}/2}\dif y\\
&\quad\geq\big(\ell\big(x_{2}-2x_{1}\big)-\ell(x_{2})\big)\int_{0}^{x_{2}}\,e^{-(y-\sqrt{I_{N}}\wh{\theta}_{N})^{2}/2}\dif y-\int_{x_{1}}^{\infty}\ell(y)\,e^{-(y-\sqrt{I_{N}}\wh{\theta}_{N})^{2}/2}\dif y\\
&\quad\geq\big(\ell(x_{1})-\ell(x_{2})\big)\int_{0}^{x_{2}}\,e^{-(y-\sqrt{I_{N}}\wh{\theta}_{N})^{2}/2}\dif y-\int_{\bR_{+}}\ell(z+x_{1})\,e^{-(z+x_{1}-\sqrt{I_{N}}\wh{\theta}_{N})^{2}/2}\dif z.
\end{align*}
Hence, for any $N\in\bN$, $\beta\geq 2x_{1}/\sqrt{I_{N}}$, and $\omega\in\Omega$ so that $\wh{\theta}_{N}(\omega)<0$,
\begin{align*}
&\int_{-\sqrt{I_{N}}\wh{\theta}_{N}(\omega)}^{\infty}\bigg(\ell\Big(\lambda+\sqrt{I_{N}}\big(\wh{\theta}_{N}(\omega)-\beta\big)\Big)-\ell\big(\lambda+\sqrt{I_{N}}\wh{\theta}_{N}(\omega)\big)\bigg)\,e^{-\lambda^{2}/2}\dif\lambda\\
&\quad\geq\big(\ell(x_{1})-\ell(x_{2})\big)x_{2}\,e^{-(x_{2}-\sqrt{I_{N}}\wh{\theta}_{N}(\omega))^{2}/2}-e^{-(x_{1}-\sqrt{I_{N}}\wh{\theta}_{N}(\omega))^{2}/2}\int_{\bR_{+}}\ell(z+x_{1})\,e^{-z^{2}/2}\dif z.
\end{align*}
Since $x_{1}>x_{2}>0$, this implies that there exists $K_{1}=K_{1}(x_{1},x_{2})>0$ such that, for any $\omega\in\Omega$ with $\sqrt{I_{N}}\wh{\theta}_{N}(\omega)\leq -K_{1}$, the right-hand side of the above inequality is strictly positive. Hence, for any $N\in\bN$ and $\beta\geq 2x_{1}/\sqrt{I_{N}}$, we obtain that
\begin{align*}
&\1_{\{\sqrt{I_{N}}\wh{\theta}_{N}\leq -K_{1}\}}\int_{-\sqrt{I_{N}}\wh{\theta}_{N}}^{\infty}\ell\Big(\lambda+\sqrt{I_{N}}\big(\wh{\theta}_{N}-\beta\big)\Big)\,e^{-\lambda^{2}/2}\dif\lambda\\
&\quad>\1_{\{\sqrt{I_{N}}\wh{\theta}_{N}\leq -K_{1}\}}\int_{-\sqrt{I_{N}}\wh{\theta}_{N}}^{\infty}\ell\big(\lambda+\sqrt{I_{N}}\wh{\theta}_{N}\big)\,e^{-\lambda^{2}/2}\dif\lambda,
\end{align*}
which, together with \eqref{eq:WTBetaNUnifPrior}, implies that
\begin{align*}
\1_{\{\sqrt{I_{N}}\wh{\theta}_{N}\leq -K_{1}\}}\,\wt{\beta}_{N}\leq\1_{\{\sqrt{I_{N}}\wh{\theta}_{N}\leq -K_{1}\}}\frac{2x_{1}}{\sqrt{I_{N}}},\quad\text{for any }N\in\bN.
\end{align*}
Since $\wt{\beta}_{N}\geq 0$, we conclude that there exists $N_{1}=N_{1}(x_{1})\in\bN$ so that $2x_{1}/\sqrt{I_{N_{1}}}\leq\inf_{\theta\in B}\theta$, and that, for any $\theta_{0}\in B$ and $N\geq N_{1}$,
\begin{align}\label{eq:UpperBoundWTBetaNMLEIni1}
\1_{\{\sqrt{I_{N}}\wh{\theta}_{N}\leq -K_{1}\}}\sqrt{I_{N}}\big|\wt{\beta}_{N}-\theta_{0}\big|\leq\1_{\{\sqrt{I_{N}}\wh{\theta}_{N}\leq -K_{1}\}}\sqrt{I_{N}}\,\theta_{0}\leq\1_{\{\sqrt{I_{N}}\wh{\theta}_{N}\leq -K_{1}\}}\sqrt{I_{N}}\big|\wh{\theta}_{N}-\theta_{0}\big|.
\end{align}

It remains to prove \eqref{eq:UpperBoundWTBetaNMLEIni}, when $\sqrt{I_{N}}\wh{\theta}_{N}> -K_{1}$. Since $\ell\in\mathbf{W}_{e,2}$ is unbounded, there exists $\wt{K}_{2}>K_{1}$ large enough such that
\begin{align*}
\int_{K_{1}}^{\wt{K}_{2}}\Big(\ell\big(\wt{K}_{2}\big)-\ell(\lambda)\Big)\,e^{-\lambda^{2}/2}\dif\lambda>\int_{[-\wt{K}_{2},\wt{K}_{2}]^{c}}\ell(\lambda)\,e^{-\lambda^{2}/2}\dif\lambda.
\end{align*}
Recalling that $\ell$ is symmetric, for any $r<-2\wt{K}_{2}$ and $N\in\bN$, we deduce that
\begin{align*}
\1_{\{\sqrt{I_{N}}\wh{\theta}_{N}\geq -K_{1}\}}\int_{-\sqrt{I_{N}}\wh{\theta}_{N}}^{\infty}\!\ell(\lambda+r)\,e^{-\lambda^{2}/2}\dif\lambda &\geq\1_{\{\sqrt{I_{N}}\wh{\theta}_{N}\geq -K_{1}\}}\int_{(-\sqrt{I_{N}}\wh{\theta}_{N})\vee(-\wt{K}_{2})}^{\wt{K}_{2}}\ell(\lambda+r)\,e^{-\lambda^{2}/2}\dif\lambda\\
&\geq\1_{\{\sqrt{I_{N}}\wh{\theta}_{N}\geq -K_{1}\}}\ell\big(\wt{K}_{2}\big)\int_{(-\sqrt{I_{N}}\wh{\theta}_{N})\vee(-\wt{K}_{2})}^{\wt{K}_{2}}e^{-\lambda^{2}/2}\dif\lambda\\
&>\1_{\{\sqrt{I_{N}}\wh{\theta}_{N}\geq -K_{1}\}}\int_{(-\sqrt{I_{N}}\wh{\theta}_{N})\vee(-\wt{K}_{2})}^{\wt{K}_{2}}\ell(\lambda)\,e^{-\lambda^{2}/2}\dif\lambda\\
&\quad +\1_{\{\sqrt{I_{N}}\wh{\theta}_{N}\geq -K_{1}\}}\int_{[-\wt{K}_{2},\wt{K}_{2}]^{c}}\ell(\lambda)\,e^{-\lambda^{2}/2}\dif\lambda\\
&> \1_{\{\sqrt{I_{N}}\wh{\theta}_{N}\geq -K_{1}\}}\int_{-\sqrt{I_{N}}\wh{\theta}_{N}}^{\infty}\ell(\lambda)\,e^{-\lambda^{2}/2}\dif\lambda,
\end{align*}
which, combined  with \eqref{eq:WTBetaNUnifPrior}, implies that
\begin{align}\label{eq:UpperBoundWTBetaNWHThetaNSmall}
\1_{\{\sqrt{I_{N}}\wh{\theta}_{N}\geq -K_{1}\}}\sqrt{I_{N}}\big(\wh{\theta}_{N}-\wt{\beta}_{N}\big)\geq\1_{\{\sqrt{I_{N}}\wh{\theta}_{N}\geq -K_{1}\}}\big(-2\wt{K}_{2}\big),\quad\text{for any }N\in\bN.
\end{align}
Using similar arguments as above, we can show that there exists $\overline{K}_{2}>K_{1}$ large enough, such that
\begin{align}\label{eq:LowerBoundWTBetaNWHThetaNSmall}
\1_{\{\sqrt{I_{N}}\wh{\theta}_{N}\geq -K_{1}\}}\sqrt{I_{N}}\big(\wh{\theta}_{N}-\wt{\beta}_{N}\big)\leq 2\overline{K}_{2}\1_{\{\sqrt{I_{N}}\wh{\theta}_{N}\geq -K_{1}\}},\quad\text{for any }N\in\bN.
\end{align}
Finally, by combining \eqref{eq:UpperBoundWTBetaNWHThetaNSmall} and \eqref{eq:LowerBoundWTBetaNWHThetaNSmall} and letting $K=\max(2\wt{K}_{2},2\overline{K}_{2})$, we obtain that, for any $N\in\bN$ and $\theta_{0}\in B$,
\begin{align}\label{eq:UpperBoundWTBetaNMLEIni2}
\1_{\{\sqrt{I_{N}}\wh{\theta}_{N}\geq -K_{1}\}}\sqrt{I_{N}}\big|\wt{\beta}_{N}-\theta_{0}\big|\leq\1_{\{\sqrt{I_{N}}\wh{\theta}_{N}\geq -K_{1}\}}\Big(\sqrt{I_{N}}\big|\wh{\theta}_{N}-\theta_{0}\big|+K\Big).
\end{align}
Combining \eqref{eq:UpperBoundWTBetaNMLEIni1} and \eqref{eq:UpperBoundWTBetaNMLEIni2} clearly leads to \eqref{eq:UpperBoundWTBetaNMLEIni}. The proof is now complete.
\end{proof}

\begin{remark}\label{rem:BayeUANExp}
The symmetry of the loss function $\ell$ is again not essential for the validity of Proposition~\ref{cor:BayeUANExp}, and can be relaxed to the non-increasing monotonicity on $(-\infty,0)$ with a more technical proof. Moreover, we also conjecture that the asymptotic efficiency for $\wt{\beta}_{N}$ remains valid with any prior $\varrho\in\cQ_{e,2}$, and the detail proof will be given elsewhere.

Finally, while we have extended the choice of loss function $\ell$ to the class $\mathbf{W}_{e,2}$ with exponential growth in the definition of $\wt{\beta}_{N}$ for Theorem~\ref{thm:BayeUANExp} and Proposition~\ref{cor:BayeUANExp}, we still keep $\varpi\in\mathbf{W}_{p}$ in the definition of asymptotic efficiency for Corollary \ref{cor:BayeUANExp}. The validity of the asymptotic efficiency for $\wt{\beta}_{N}$ with more general $\varpi\in\mathbf{W}_{e,2}$ is yet to be studied.
\end{remark}

\subsection{Asymptotic properties of $\wh{\beta}_{N}$}\label{sec:AsymHatBetaN}

In this section we will study the asymptotic properties of the Bayesian estimator $\wh\beta_N$. While we will investigate only (strong) consistency, and asymptotic normality of  $\wh\beta_N$, we recall that $\wh\beta_N$ is defined for larger class of loss functions $\mathbf{W}'$, in comparison to $\widetilde{\beta}_N$ from previous section.

When $\wh{\beta}_{N}$ is well defined for all $N\in\bN$, the following theorem provides a sufficient condition for the consistency and asymptotic normality of $\wh{\beta}_{N}$.

\begin{theorem}\label{thm:BayeEst}
Assume that the prior density $\varrho\in\cQ$ is positive and continuous in a neighborhood of $\theta_{0}$, and that $\wh{\beta}_{N}$ is well-defined with respect to a loss function $\ell\in\mathbf{W}'$, for every $N\in\bN$. Moreover, assume that there exists $\{a_{N}\}_{N\in\bN}\subset\bR_{+}$, a test function $f$ as in Theorem \ref{thm:BvM} satisfying conditions (C1) and (C2), and another loss function $\wt{\ell}$, such that
\begin{itemize}
\item [(i)] $\,\displaystyle{a_{N}\,\ell\bigg(\frac{\lambda}{\sqrt{I_{N}}}\bigg)\leq f(\lambda)},\,\,$ for any $\lambda\in\bR$ and $N\in\bN$;
\item [(ii)] $\,\displaystyle{\lim_{N\rightarrow\infty}\sup_{\lambda\in B}\left|a_{N}\,\ell\bigg(\frac{\lambda}{\sqrt{I_{N}}}\bigg)-\wt{\ell}(\lambda)\right|=0},\,\,$ for any compact set $B\subset\bR$;
\item [(iii)] $\,\displaystyle{r\mapsto\int_{\bR}\wt{\ell}(\lambda+r)\,e^{-\lambda^{2}/2}\dif\lambda}$ has a strict minimum at $r=0$.
\end{itemize}
Then,
\begin{itemize}
\item [(a)] $\,\displaystyle{\lim_{N\rightarrow\infty}\sqrt{I_{N}}\big(\wh{\beta}_{N}-\wh{\theta}_{N}\big)=0}$, $\,\bP-\text{a.s.}$;
\item [(b)] $\,\displaystyle{\lim_{N\rightarrow\infty}a_{N}\int_{\bR_{+}}\ell\big(\theta-\wh{\beta}_{N}\big)\,p(\theta\,|\,U_{N})\dif\theta=\int_{\bR}\wt{\ell}(\lambda)\,\frac{e^{-\lambda^{2}/2}}{\sqrt{2\pi}}\dif\lambda}$, $\,\bP-\text{a.s.}$.
\end{itemize}
In particular, $\wh{\beta}_N$ is strongly consistent and asymptotically normal, as $N\rightarrow\infty$, namely
\begin{align*}
\wh{\beta}_{N}\rightarrow\theta_{0},\,\,\,\bP-\text{a.s.}\quad\text{and}\quad\sqrt{I_{N}}\big(\wh{\beta}_{N}-\theta_{0}\big)\cd\cN(0,1).
\end{align*}
\end{theorem}

\begin{proof}
The strong consistency and asymptotic normality of $\wh{\beta}_{N}$ are immediate consequences of part (a) together with the strong consistency and asymptotic normality of $\wh{\theta}_{N}$ (recalling \eqref{eq:MLEIniStrConst} and \eqref{eq:MLEIniAsymNormal}).

The proof of  (a) and (b) is split in four steps.

\smallskip
\noindent
\textit{Step 1.} We will first show that
\begin{align}\label{eq:LimsupaN}
\limsup_{N\rightarrow\infty}a_{N}\int_{\bR_{+}}\ell\big(\theta-\wh{\beta}_{N}\big)\,p(\theta\,|\,U_{N})\dif\theta\leq\int_{\bR}\wt{\ell}(\lambda)\,\frac{e^{-\lambda^{2}/2}}{\sqrt{2\pi}}\dif\lambda,\quad\bP-\text{a.s.}.
\end{align}
By the definition of $\wh{\beta}_{N}$,
\begin{align*}
\int_{\bR_{+}}\ell\big(\theta-\wh{\beta}_{N}\big)\,p(\theta\,|\,U_{N})\dif\theta\leq\int_{\bR_{+}}\ell\big(\theta-\wh{\theta}_{N}\big)\,p(\theta\,|\,U_{N})\dif\theta=\int_{\bR}\ell\bigg(\frac{\lambda}{\sqrt{I_{N}}}\bigg)\,\wt{p}(\lambda\,|\,U_{N})\dif\lambda.
\end{align*}
Hence, to prove \eqref{eq:LimsupaN}, it suffices to show that
\begin{align}\label{eq:LimaN}
\lim_{N\rightarrow\infty}a_{N}\int_{\bR}\ell\bigg(\frac{\lambda}{\sqrt{I_{N}}}\bigg)\,\wt{p}(\lambda\,|\,U_{N})\dif\lambda=\int_{\bR}\wt{\ell}(\lambda)\,\frac{e^{-\lambda^{2}/2}}{\sqrt{2\pi}}\dif\lambda,\quad\bP-\text{a.s.}.
\end{align}
Using conditions (i) and (ii), we obtain that $\wt{\ell}(\lambda)\leq f(\lambda)$, for any $\lambda\in\bR$. Therefore, by conditions (i) and (ii), Theorem \ref{thm:BvM}, and the dominated convergence theorem, we deduce that
\begin{align*}
&\lim_{N\rightarrow\infty}\left|a_{N}\int_{\bR}\ell\bigg(\frac{\lambda}{\sqrt{I_{N}}}\bigg)\wt{p}(\lambda\,|\,U_{N})\dif\lambda-\int_{\bR}\wt{\ell}(\lambda)\,\frac{e^{-\lambda^{2}/2}}{\sqrt{2\pi}}\dif\lambda\right|\\
&\quad\leq\lim_{N\rightarrow\infty}\int_{\bR}\left|a_{N}\ell\bigg(\frac{\lambda}{\sqrt{I_{N}}}\bigg)-\wt{\ell}(\lambda)\right|\bigg|\wt{p}(\lambda\,|\,U_{N})-\frac{e^{-\lambda^{2}/2}}{\sqrt{2\pi}}\bigg|\dif\lambda+\lim_{N\rightarrow\infty}\int_{\bR}\wt{\ell}(\lambda)\,\bigg|\wt{p}(\lambda\,|\,U_{N})-\frac{e^{-\lambda^{2}/2}}{\sqrt{2\pi}}\bigg|\dif\lambda\\
&\qquad +\lim_{N\rightarrow\infty}\int_{\bR}\left|a_{N}\ell\bigg(\frac{\lambda}{\sqrt{I_{N}}}\bigg)-\wt{\ell}(\lambda)\right|\frac{e^{-\lambda^{2}/2}}{\sqrt{2\pi}}\dif\lambda\\
&\quad\leq\lim_{N\rightarrow\infty}3\int_{\bR}f(\lambda)\bigg|\wt{p}(\lambda\,|\,U_{N})-\frac{e^{-\lambda^{2}/2}}{\sqrt{2\pi}}\bigg|\dif\lambda+\int_{\bR}\left(\lim_{N\rightarrow\infty}\bigg|a_{N}\ell\bigg(\frac{\lambda}{\sqrt{I_{N}}}\bigg)-\wt{\ell}(\lambda)\bigg|\right)\frac{e^{-\lambda^{2}/2}}{\sqrt{2\pi}}\dif\lambda\\
&\quad\leq 0+\lim_{K\rightarrow\infty}\int_{-K}^{K}\left(\lim_{N\rightarrow\infty}\bigg|a_{N}\ell\bigg(\frac{\lambda}{\sqrt{I_{N}}}\bigg)-\wt{\ell}(\lambda)\bigg|\right)\frac{e^{-\lambda^{2}/2}}{\sqrt{2\pi}}\dif\lambda\\
&\quad\leq\lim_{K\rightarrow\infty}\left(\lim_{N\rightarrow\infty}\sup_{\lambda\in[-K,K]}\bigg|a_{N}\ell\bigg(\frac{\lambda}{\sqrt{I_{N}}}\bigg)-\wt{\ell}(\lambda)\bigg|\right)\int_{-K}^{K}\frac{e^{-\lambda^{2}/2}}{\sqrt{2\pi}}\dif\lambda=0,
\end{align*}
which completes the proof of \eqref{eq:LimaN}, and therefore, of \eqref{eq:LimsupaN}.

\smallskip
\noindent
\textit{Step 2.} Next, we will show that the sequence of random variables $Y_{N}:=\sqrt{I_{N}}(\wh{\beta}_{N}-\wh{\theta}_{N})$, $N\in\bN$, are uniformly bounded, $\bP-$a.s.. That is, for $\bP-\text{a.s.}\,\,\omega$, there exists $K(\omega)\in(0,\infty)$, such that $|Y_{N}(\omega)|\leq K(\omega)$ for all $N\in\bN$.

For any $\omega\in A:=\{\omega\in\Omega:\,\limsup_{N\rightarrow\infty}|Y_{N}(\omega)|=\infty\}$ and any $K\in\bN$, there exists an increasing sequence of integers $N_{j}=N_{j}(\omega,K)$, $j\in\bN$, such that $N_{j}\uparrow\infty$, as $j\rightarrow\infty$, and $|Y_{N_{j}}(\omega)|\geq K$, for any $j\in\bN$. Consequently,
\begin{align}
\int_{\bR_{+}}\ell\big(\theta-\wh{\beta}_{N_{j}}(\omega)\big)\,p\big(\theta\,|\,U_{N_{j}}\big)(\omega)\dif\theta &=\int_{\bR}\ell\bigg(\frac{\lambda+Y_{N_{j}}(\omega)}{\sqrt{I_{N_{j}}}}\bigg)\,\wt{p}\big(\lambda\,|\,U_{N_{j}}\big)(\omega)\dif\lambda\nonumber\\
\label{eq:KLowerBound} &\geq\int_{-K}^{K}\ell\bigg(\frac{\lambda+K}{\sqrt{I_{N_{j}}}}\bigg)\,\wt{p}\big(\lambda\,|\,U_{N_{j}}\big)(\omega)\dif\lambda,
\end{align}
where we used \eqref{eq:MonotoneLossFunt} in the last inequality. On the other hand, since $\wt{\ell}$ is locally bounded, the function $\lambda\mapsto\1_{[-K,K]}(\lambda)\wt{\ell}(\lambda+K)$ is bounded on $\bR$ and thus satisfies conditions (C1) and (C2). Hence, by Theorem \ref{thm:BvM} and condition (ii), as $j\rightarrow\infty$,
\begin{align}
&\left|a_{N_{j}}\int_{-K}^{K}\ell\bigg(\frac{\lambda+K}{\sqrt{I_{N_{j}}}}\bigg)\,\wt{p}\big(\lambda\,|\,U_{N_{j}}\big)\dif\lambda-\int_{-K}^{K}\wt{\ell}(\lambda+K)\,\frac{e^{-\lambda^{2}/2}}{\sqrt{2\pi}}\dif\lambda\right|\nonumber\\
&\quad\leq\int_{-K}^{K}\left|a_{N_{j}}\ell\bigg(\frac{\lambda+K}{\sqrt{I_{N_{j}}}}\bigg)-\wt{\ell}(\lambda+K)\right|\wt{p}\big(\lambda\,|\,U_{N_{j}}\big)\dif\lambda+\int_{-K}^{K}\wt{\ell}(\lambda+K)\bigg|\wt{p}\big(\lambda\,|\,U_{N_{j}}\big)-\frac{e^{-\lambda^{2}/2}}{\sqrt{2\pi}}\bigg|\dif\lambda\nonumber\\
\label{eq:LimitaNellK} &\quad\leq\sup_{\lambda\in[-2K,2K]}\!\left|a_{N_{j}}\ell\bigg(\frac{\lambda}{\sqrt{I_{N_{j}}}}\bigg)\!-\!\wt{\ell}(\lambda)\right|+\!\int_{-K}^{K}\!\wt{\ell}(\lambda\!+\!K)\bigg|\wt{p}\big(\lambda\,|\,U_{N_{j}}\big)\!-\!\frac{e^{-\lambda^{2}/2}}{\sqrt{2\pi}}\bigg|\dif\lambda\rightarrow 0,\,\,\bP-\text{a.s.}.
\end{align}
Let $B$ denote the exceptional subset of $\Omega$ in which the limit in \eqref{eq:LimitaNellK} does not hold, then $\bP(B)=0$. Without loss of generality, assume that $A\setminus B\neq\emptyset$. By \eqref{eq:KLowerBound} and \eqref{eq:LimitaNellK}, for any $\omega\in A\setminus B$,
\begin{align*}
\limsup_{j\rightarrow\infty}a_{N_{j}}\int_{\bR_{+}}\ell\big(\theta-\wh{\beta}_{N_{j}}(\omega)\big)\,p\big(\theta\,|\,U_{N_{j}}\big)(\omega)\dif\theta &\geq\lim_{j\rightarrow\infty}\int_{-K}^{K}\ell\bigg(\frac{\lambda+K}{\sqrt{I_{N_{j}}}}\bigg)\,\wt{p}\big(\lambda\,|\,U_{N_{j}}\big)(\omega)\dif\lambda\\
&=\int_{-K}^{K}\wt{\ell}(\lambda+K)\,\frac{e^{-\lambda^{2}/2}}{\sqrt{2\pi}}\dif\lambda.
\end{align*}
Since $K\in\bN$ is arbitrary, by condition (iii) and monotone convergence theorem, we have
\begin{align*}
\limsup_{j\rightarrow\infty}a_{N_{j}}\!\int_{\bR_{+}}\!\!\ell\big(\theta-\wh{\beta}_{N_{j}}(\omega)\big)\,p\big(\theta\,|\,U_{N_{j}}\big)(\omega)\dif\theta\geq\lim_{K\rightarrow\infty}\int_{-K}^{K}\wt{\ell}(\lambda+K)\,\frac{e^{-\lambda^{2}/2}}{\sqrt{2\pi}}d\lambda>\!\int_{\bR}\wt{\ell}(\lambda)\,\frac{e^{-\lambda^{2}/2}}{\sqrt{2\pi}}\dif\lambda.
\end{align*}
In view of \eqref{eq:LimsupaN}, we must have $\bP(A\setminus B)=0$ so that $\bP(A)=0$, completing the proof of Step 2.

\smallskip
\noindent
\textit{Step 3.} We now prove (a) and (b). By Step 2, for $\bP-\text{a.s.}\,\,\omega$, there exists $N_{j}=N_{j}(\omega)\in\bN$, $j\in\bN$, such that the sequence $(Y_{N_{j}}(\omega))_{j\in\bN}$ is convergent, as $j\rightarrow\infty$, and we denote its limit by $Y^{*}(\omega)$. For any $K>0$,
\begin{align}
a_{N_{j}}\int_{\bR_{+}}\ell\big(\theta-\wh{\beta}_{N_{j}}(\omega)\big)\,p\big(\theta\,|\,U_{N_{j}}\big)(\omega)\dif\theta&=a_{N_{j}}\int_{\bR}\ell\bigg(\frac{\lambda+Y_{N_{j}}(\omega)}{\sqrt{I_{N_{j}}}}\bigg)\,\wt{p}\big(\lambda\,|\,U_{N_{j}}\big)\dif\lambda\nonumber\\
\label{eq:KLowerboundUNj} &\geq a_{N_{j}}\int_{-K}^{K}\ell\bigg(\frac{\lambda+Y_{N_{j}}(\omega)}{\sqrt{I_{N_{j}}}}\bigg)\,\wt{p}\big(\lambda\,|\,U_{N_{j}}\big)\dif\lambda.
\end{align}
Moreover, for $j\in\bN$ large enough,
\begin{align}
&\left|a_{N_{j}}\int_{-K}^{K}\ell\bigg(\frac{\lambda+Y_{N_{j}}(\omega)}{\sqrt{I_{N_{j}}}}\bigg)\,\wt{p}\big(\lambda\,|\,U_{N_{j}}\big)\dif\lambda-\int_{-K}^{K}\wt{\ell}(\lambda+Y^{*}(\omega))\,\frac{e^{-\lambda^{2}/2}}{\sqrt{2\pi}}\dif\lambda\right|\nonumber\\
&\quad\leq\int_{-K}^{K}\left|a_{N_{j}}\ell\bigg(\frac{\lambda+Y_{N_{j}}(\omega)}{\sqrt{I_{N_{j}}}}\bigg)-\wt{\ell}\big(\lambda+Y_{N_{j}}(\omega)\big)\right|\wt{p}\big(\lambda\,|\,U_{N_{j}}\big)(\omega)\dif\lambda\nonumber\\
&\qquad +\!\int_{-K}^{K}\!\wt{\ell}\big(\lambda\!+\!Y_{N_{j}}(\omega)\big)\bigg|\wt{p}\big(\lambda\,|\,U_{N_{j}}\big)(\omega)\!-\!\frac{e^{-\lambda^{2}/2}}{\sqrt{2\pi}}\bigg|\dif\lambda+\!\int_{-K}^{K}\!\Big|\wt{\ell}\big(\lambda\!+\!Y_{N_{j}}(\omega)\big)\!-\!\wt{\ell}(\lambda\!+\!Y^{*}(\omega))\Big|\frac{e^{-\lambda^{2}/2}}{\sqrt{2\pi}}\dif\lambda\nonumber\\
&\quad\leq\int_{-K}^{K}\left|a_{N_{j}}\ell\bigg(\frac{\lambda+Y_{N_{j}}(\omega)}{\sqrt{I_{N_{j}}}}\bigg)-\wt{\ell}\big(\lambda+Y_{N_{j}}(\omega)\big)\right|\wt{p}\big(\lambda\,|\,U_{N_{j}}\big)(\omega)\dif\lambda\nonumber\\
&\qquad +\!\int_{-K}^{K}\!\Big(\tilde{\ell}(\lambda+Y^{*}(\omega)+1)\1_{[0,\infty)}(\lambda)+\wt{\ell}(\lambda+Y^{*}(\omega)-1)\1_{(-\infty,0)}(\lambda)\Big)\bigg|\wt{p}\big(\lambda\,|\,U_{N_{j}}\big)(\omega)-\frac{e^{-\lambda^{2}/2}}{\sqrt{2\pi}}\bigg|\dif\lambda\nonumber\\
\label{eq:DecompaNU*} &\qquad +\int_{-K}^{K}\Big|\wt{\ell}\big(\lambda+Y_{N_{j}}(\omega)\big)-\wt{\ell}(\lambda+Y^{*}(\omega))\Big|\,\frac{e^{-\lambda^{2}/2}}{\sqrt{2\pi}}\dif\lambda.
\end{align}
By the same argument as in \eqref{eq:LimitaNellK}, the first two integrals in \eqref{eq:DecompaNU*} vanish as $j\rightarrow\infty$, for $\bP-\text{a.s.}\,\,\omega$. Moreover, the monotonicity property \eqref{eq:MonotoneLossFunt} for $\wt{\ell}$ implies that it is almost surely (with respect to the Lebesgue measure) continuous on $\bR$. Hence, conditions (i) and (ii) (which implies that $\wt{\ell}$ is bounded by $f$), together with the dominated convergence theorem, imply that the last integral in \eqref{eq:DecompaNU*} vanishes as $j\rightarrow\infty$, for $\bP-\text{a.s.}\,\,\omega$. Therefore, for $\bP-\text{a.s.}\,\,\omega$, we have
\begin{align}\label{eq:LimitaNellKUNj}
\lim_{j\rightarrow\infty}a_{N_{j}}\int_{-K}^{K}\ell\bigg(\frac{\lambda+Y_{N_{j}}(\omega)}{\sqrt{I_{N_{j}}}}\bigg)\,\wt{p}\big(\lambda\,|\,U_{N_{j}}\big)\dif\lambda=\int_{-K}^{K}\wt{\ell}(\lambda+Y^{*}(\omega))\,\frac{e^{-\lambda^{2}/2}}{\sqrt{2\pi}}\dif\lambda.
\end{align}
Combining \eqref{eq:KLowerboundUNj} and \eqref{eq:LimitaNellKUNj}, and noting that $K>0$ is arbitrary, we obtain that
\begin{align}
\liminf_{j\rightarrow\infty}a_{N_{j}}\int_{\bR_{+}}\ell\big(\theta-\wh{\beta}_{N_{j}}\big)\,p\big(\theta\,|\,U_{N_{j}}\big)\dif\theta &\geq\int_{\bR}\wt{\ell}(\lambda+Y^{*}(\omega))\,\frac{e^{-\lambda^{2}/2}}{\sqrt{2\pi}}\dif\lambda\nonumber\\
\label{eq:LiminfaN} &\geq\int_{\bR}\wt{\ell}(\lambda)\,\frac{e^{-\lambda^{2}/2}}{\sqrt{2\pi}}\dif\lambda,\quad\bP-\text{a.s.},
\end{align}
where we have used condition (iii) in the second inequality.

Now assume that $\bP(Y^{*}\neq 0)>0$. Then inequality on \eqref{eq:LiminfaN} would be strict in $\{Y^{*}\neq 0\}$ by assumption (iii), which is clearly a contradiction to \eqref{eq:LimsupaN}. Therefore, $\bP(Y^{*}\neq 0)=0$. Combining \eqref{eq:LimsupaN} in Step 1 with \eqref{eq:LiminfaN} completes the proof of (b).

To prove (a), it remains to show that every subsequence of $(Y_{N})_{N\in\bN}$ converges to $0$ with probability one. Indeed, assume the contrary. Applying similar arguments as in Step~3 to this exceptional subsequence leads to a contradiction to \eqref{eq:LimsupaN}.

The proof is now complete.
\end{proof}

\begin{remark}\label{rem:LossFunt}
Here we present an example of loss function in $\mathbf{W}_{e,2}$ that satisfies conditions (i)$-$(iii) in Theorem \ref{thm:BayeEst} above. Let $\ell(x)=\exp(|x|^{r})-1$ with $r\in(0,2)$, and $a_{N}=I_{N}^{r/2}$. Without loss of generality, we assume $I_{N}\geq 1$ for all $N\in\bN$. Note that
\begin{align*}
a_{N}\ell\bigg(\frac{\lambda}{\sqrt{I_{N}}}\bigg)=I_{N}^{r/2}\Big(e^{|\lambda|^{r}I_{N}^{-r/2}}-1\Big)\leq |\lambda|^{r}e^{|\lambda|^{r}I_{N}^{-r/2}}\leq |\lambda|^{r}e^{|\lambda|^{r}}=:f(\lambda),
\end{align*}
which satisfies conditions (C1) and (C2), due to Remark~\ref{rem:ExpGrowPriorLossFunt}. Hence, condition (i) is valid. Moreover, let $\wt{\ell}(x)=|x|^{r}$, which clearly satisfies condition (iii). For any $K>0$ and $\lambda\in[-K,K]$,
\begin{align*}
&\left|a_{N}\ell\bigg(\frac{\lambda}{\sqrt{I_{N}}}\bigg)-\wt{\ell}(\lambda)\right|=I_{N}^{r/2}\Big(e^{|\lambda|^{r}I_{N}^{-r/2}}-1-|\lambda|^{r}I_{N}^{-r/2}\Big)\leq K^{2r}I_{N}^{-r/2}e^{K^{r}I_{N}^{-r/2}}\rightarrow 0,
\end{align*}
as $N\rightarrow\infty$, which shows the validity of condition (ii).
\end{remark}

\section{Numerical example}\label{sec:NumericalEx}

In this section, we provide an illustrative numerical example of the asymptotics of the MLE and the Bayesian estimators derived in the previous sections. Specifically, we consider the following equation
\begin{align}\label{eq:exSPDE}
\begin{cases}
\dif u(t,x)+\theta_{0}(-u_{xx}(t,x))\dif t=\sigma\sum_{k=1}^{\infty}u_{k}(t)h_{k}(x)k^{\alpha}\dif w_{k}(t), \ t>0\\
u(0,x)=\frac{\pi^{2}}{4}-\left(x-\frac{\pi}{2}\right)^{2}, x\in[0,\pi], \\
u(t,0)=u(t,\pi)=0, t>0,
\end{cases}
\end{align}
where $h_{k}(x):=\sqrt{2/\pi}\sin(kx)$, $x\in[0,\pi]$, and $\alpha\in\bR$.  We will fix $\sigma=1$ and $T=1$ for the rest of this section. We note that in this case, $A= - \partial_{xx}$, with its eigenfunctions $h_k, \ k\in\bN$, and corresponding eigenvalues  $\mu_k = k^2, \ k\in\bN$. In view of \eqref{eq:FisherInfo}, the Fisher information $I_N \varpropto N^{5-2\alpha}$.

We will consider two set of parameters, one for which condition (E1) from Remark~\ref{rem:condExist} is satisfied, and another one that corresponds to condition (E2) from the same remark. All evaluations are performed using numerical computing environment \textsc{Matlab}, and the source codes are available from the authors upon request.

\medskip\noindent
\textbf{Parameter Set I.} We take the true value of the parameter of interest to be $\theta_{0}=0.3$, and $\alpha<1$; clearly (E1) is satisfied. When $\alpha\uparrow 1$, the rate of divergence of $I_{N}$ as $N\to\infty$, and thus the rates of convergence of both the MLE and the Bayesian estimators, become smaller, as $N\rightarrow\infty$. We will illustrate this below by considering two different regimes $\alpha=0$ and $\alpha=0.999$. We also take two different priors - uninformed and truncated normal.

We simulate $N=20$ independent Brownian sample paths $\{w_{k}(t),t\in[0,1]\}$, $k=1,\ldots,20$, with time-step of $5\times 10^{-5}$, and compute the corresponding Fourier modes $\{u_{k}(t),t\in[0,1]\}$, $k=1,\ldots,20$, using \eqref{eq:Soluk} and $\theta_0$. Then, we compute $\widehat{\theta}_N$ according to \eqref{eq:MLEIni1}, where the It\^{o} integral is approximated by a finite sum with time-step of $5\times 10^{-5}$. We also computed $\wh\theta_N$ using \eqref{eq:MLEIni2},  that yielded  same results and they will not be reported here. Moreover, we simulate the posterior density $p(\theta|U_{N})$ using \eqref{eq:PosteriorThetaDeri}, where the integral in the denominator is approximated using the `integral' \textsc{Matlab} built-in  function.

In Figures~\ref{fig1} and~\ref{fig2}, we present and compare the posterior density $p(\theta|U_{N})$ under two different prior distributions, the uniform uninformative prior on $\bR_+$, $\varrho(x)=\1_{\bR_+}(x)$ (left panels), and truncated normal $\cN(1,0.1;\bR_{+})$ (right panels), for $\alpha=0$ and $\alpha=0.999$, respectively. Also in Figures~\ref{fig1} and~\ref{fig2} we mark the MLE for $N=2$ (black square on the horizontal axis), and $N=4$ (black diamond on the horizontal axis). For both choices of $\alpha$, the posterior densities under both priors converge to the Dirac measure concentrated at the MLEs, which is consistent with the discussion in Remark \ref{rem:BvM}. Moreover, under both priors, the posterior densities with $\alpha=0$ exhibit faster convergence rates than those with $\alpha=0.999$ as expected.
\begin{figure}[!ht]
     \centering
    \begin{subfigure}[t]{0.48\textwidth}
        \centering
        \includegraphics[width=\linewidth]{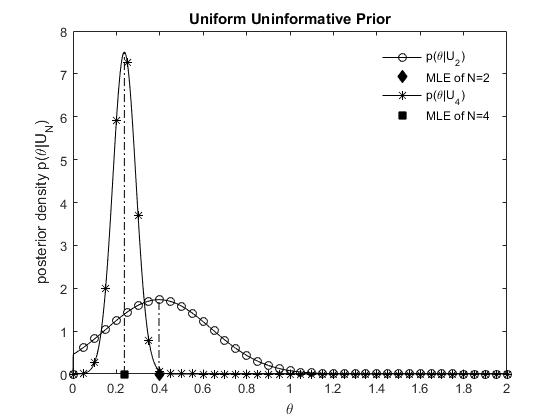}
        %\caption{Generic} \label{fig:timing1}
    \end{subfigure}
    \hfill
     \begin{subfigure}[t]{0.48\textwidth}
        \centering
        \includegraphics[width=\linewidth]{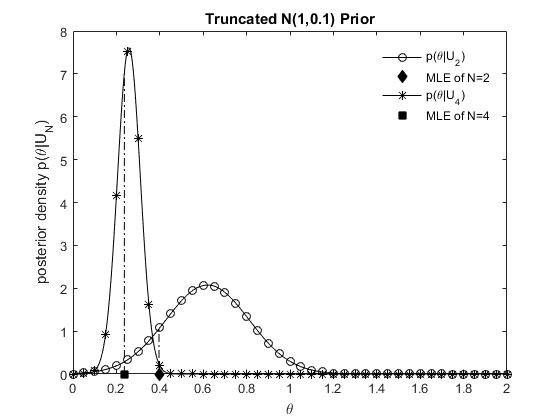}
    \end{subfigure}
    \caption{Parameter Set I. Posterior densities for $N=2,4$, two different priors, and $\alpha=0$.}
    \label{fig1}
\end{figure}
\begin{figure}[!ht]
     \centering
    \begin{subfigure}[t]{0.48\textwidth}
        \centering
        \includegraphics[width=\linewidth]{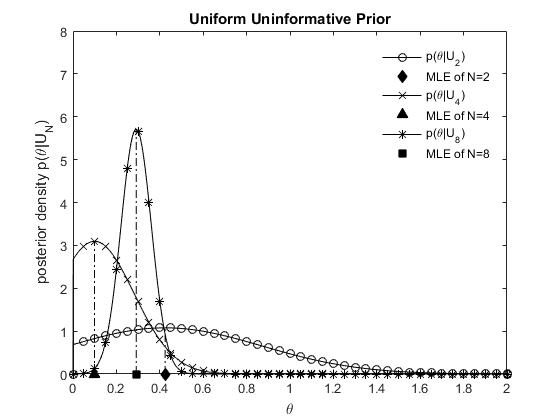}
        %\caption{Generic} \label{fig:timing1}
    \end{subfigure}
    \hfill
     \begin{subfigure}[t]{0.48\textwidth}
        \centering
        \includegraphics[width=\linewidth]{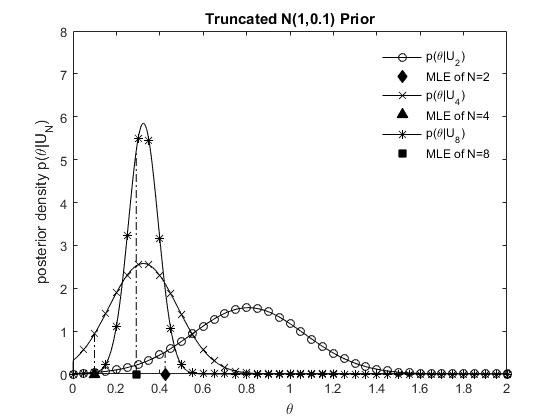}
    \end{subfigure}
    \caption{Parameter Set I. Posterior densities for $N=2,4,8$, two different priors, and $\alpha=0.999$.}
    \label{fig2}
\end{figure}

Next, we take the quadratic loss function $\ell(x)=x^{2}$, in which case $\wt{\beta}_N$ and $\widehat{\beta}_N$ are given by
\begin{align*}
\wt{\beta}_N=\widehat{\beta}_{N}=\int_{\bR_{+}}\theta\,p(\theta|U_{N})\dif\theta.
\end{align*}
Also note that such Bayesian estimator can be viewed as the conditional mean estimator. In Figures~\ref{fig3} and~\ref{fig4}, we compare $\widehat{\theta}_{N}$ with the Bayesian estimator $\widehat{\beta}_{N}$, using the two considered priors and  for $\alpha=0$ and $\alpha=0.999$. Again, while $\widehat{\theta}_N$ and the Bayesian estimators converge to the true parameter $\theta_{0}=0.3$, as the number of Fourier modes increases, for both choices of $\alpha$, the case of $\alpha=0$ tends to have a better convergence rate. Finally, in Figure~\ref{fig5}, we display the values of $\sqrt{I_N}\left|\widehat{\beta}_N-\widehat{\theta}_N\right|$ as function of $N$, that confirm the asymptotic results of Theorem~\ref{thm:BayeUANExp} and Theorem~\ref{thm:BayeEst}.

\begin{figure}[!ht]
     \centering
    \begin{subfigure}[t]{0.48\textwidth}
        \centering
        \includegraphics[width=\linewidth]{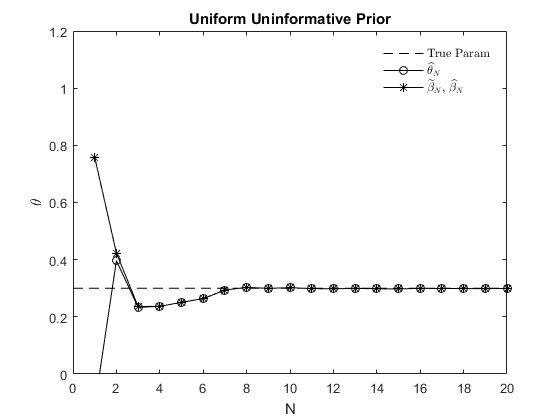}
        %\caption{Generic} \label{fig:timing1}
    \end{subfigure}
    \hfill
     \begin{subfigure}[t]{0.48\textwidth}
        \centering
        \includegraphics[width=\linewidth]{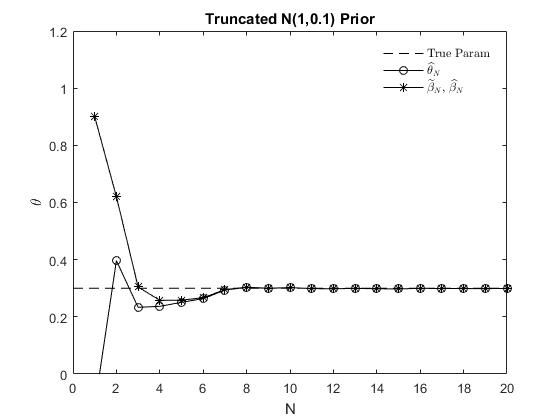}
    \end{subfigure}
    \caption{Parameter Set I. MLE vs. Bayesian Estimator with $\alpha=0$.}
    \label{fig3}
\end{figure}
\begin{figure}[!ht]
     \centering
    \begin{subfigure}[t]{0.48\textwidth}
        \centering
        \includegraphics[width=\linewidth]{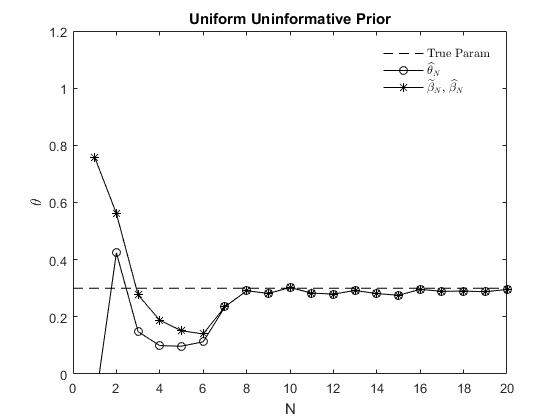}
        %\caption{Generic} \label{fig:timing1}
    \end{subfigure}
    \hfill
     \begin{subfigure}[t]{0.48\textwidth}
        \centering
        \includegraphics[width=\linewidth]{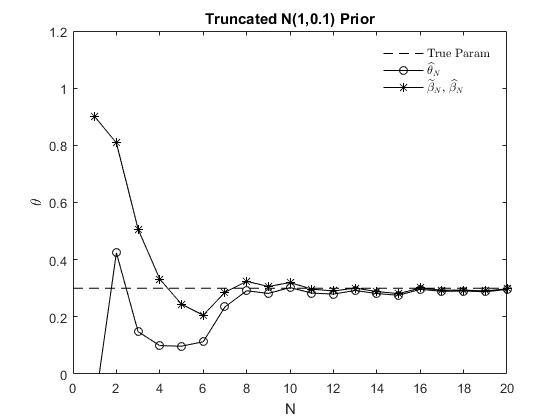}
    \end{subfigure}
    \caption{Parameter Set I. MLE vs. Bayesian Estimator with $\alpha=0.999$.}
    \label{fig4}
\end{figure}
\begin{figure}[!ht]
     \centering
    \begin{subfigure}[t]{0.48\textwidth}
        \centering
        \includegraphics[width=\linewidth]{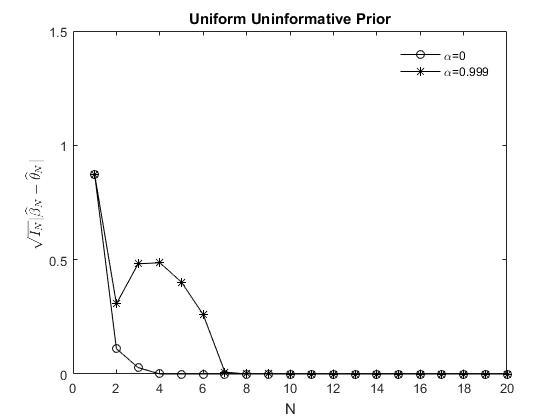}
        %\caption{Generic} \label{fig:timing1}
    \end{subfigure}
    \hfill
     \begin{subfigure}[t]{0.48\textwidth}
        \centering
        \includegraphics[width=\linewidth]{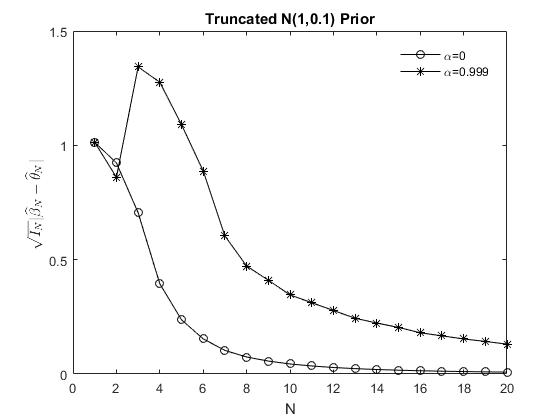}
    \end{subfigure}
    \caption{Parameter Set I. $\sqrt{I_N}\left|\widehat{\beta}_N-\widehat{\theta}_N\right|$, for two different priors, and $\alpha=0$, and $\alpha=0.999$.}
    \label{fig5}
\end{figure}

\medskip

\noindent\textbf{Parameter Set II.} In the second set of parameters, we let $\theta_0=0.505$ and $\alpha=1$, which corresponds to the case (E2) from Remark~\ref{rem:condExist}. We consider the same two priors as above. The posterior densities behave similarly as in Case I. In contrast to Case I, we consider a loss function with exponential growth, namely, $\ell(x)=\exp\left(|x|^{3/2}\right)-1$, and we compute both $\widetilde{\beta}_N$ and $\wh\beta_N$ using $\ell$.
In Figure~\ref{fig8} we plot the values of MLE and the two Bayesian estimators as function of $N$. All estimators perform well, and similarly. Although we display the results for one path, similar behavior is observed on other simulated paths of the solution. The values of scaled errors  $\sqrt{I_N}|\widetilde{\beta}_N-\widehat{\theta}_N|$ and $\sqrt{I_N}|\widehat{\beta}_N-\widehat{\theta}_N|$  are displayed in Figure~\ref{fig9}, again confirming the results of Theorem~\ref{thm:BayeUANExp} and Theorem~\ref{thm:BayeEst}.

\begin{figure}[!ht]
     \centering
    \begin{subfigure}[t]{0.48\textwidth}
        \centering
        \includegraphics[width=\linewidth]{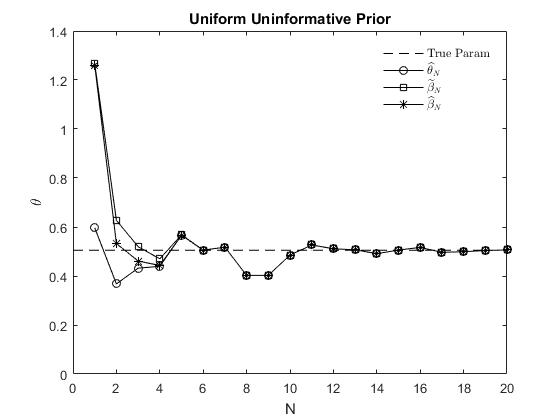}
        %\caption{Generic} \label{fig:timing1}
    \end{subfigure}
    \hfill
     \begin{subfigure}[t]{0.48\textwidth}
        \centering
        \includegraphics[width=\linewidth]{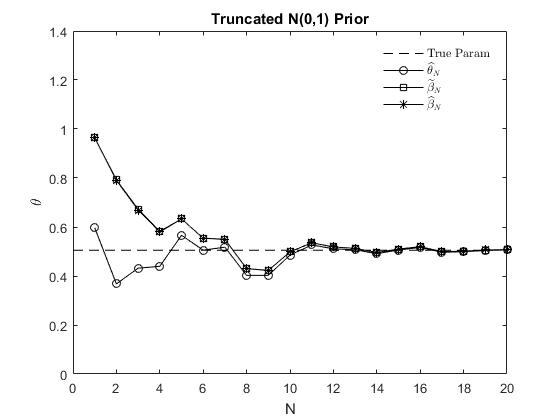}
    \end{subfigure}
    \caption{Parameter Set II. Value of $\widehat{\theta}_N$, $\widetilde{\beta}_N$, $\widehat{\beta}_N$, for two different priors, and $\alpha=1$.}
    \label{fig8}
\end{figure}

\begin{figure}[!ht]
     \centering
    \begin{subfigure}[t]{0.48\textwidth}
        \centering
        \includegraphics[width=\linewidth]{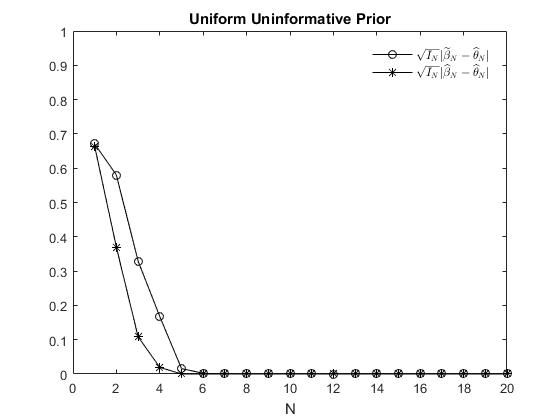}
        %\caption{Generic} \label{fig:timing1}
    \end{subfigure}
    \hfill
     \begin{subfigure}[t]{0.48\textwidth}
        \centering
        \includegraphics[width=\linewidth]{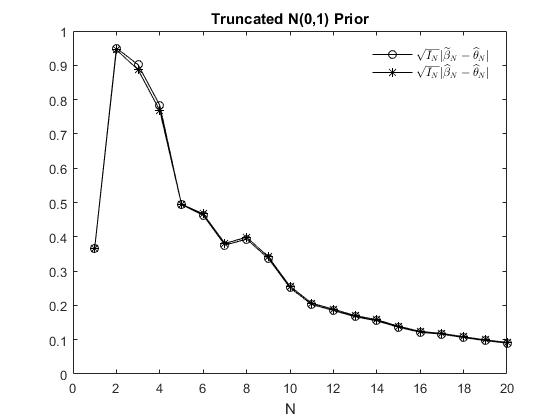}
    \end{subfigure}
    \caption{Parameter Set II.  Scaled errors $\sqrt{I_N}\left|\widetilde{\beta}_N-\widehat{\theta}_N\right|$, $\sqrt{I_N}\left|\widehat{\beta}_N-\widehat{\theta}_N\right|$, for two different priors, and $\alpha=1$}
    \label{fig9}
\end{figure}

\appendix
\section{Auxiliary Results}\label{AppendA}
For the sake of completeness, we recall a version of the strong law of large number, which will be used in the proof of the strong consistency of MLE.
We refer the reader to~\cite[Theorem IV.3.2]{ShiryaevBookProbability} for its detail proof.
\begin{theorem}[Strong Law of Large Numbers]\label{thm:SLLN}
Let $\{\xi_n\}_{n\in\bN}$ be a sequence of independent random variables with finite second moments. Let $\{b_n\}_{n\in\bN}$ be a sequence of non-decreasing real numbers such that $\lim_{n\to\infty}b_{n}=\infty$, and $\sum_{n=1}^{\infty} \Var(\xi_{n})/b_{n}^{2}<\infty.$
Then,
\begin{align*}
\frac{\sum_{k=1}^{n}\left(\xi_{k}-\bE(\xi_{k})\right)}{b_{n}}\rightarrow 0,\quad\bP-\text{a.s.}\,,\quad n\rightarrow\infty.
\end{align*}
\end{theorem}
Also here we present a simple technical lemma used in Section~\ref{sec:MLE}.

\begin{lemma}\label{lem:techlem1}
If the sequence $\{a_n\}_{n\in\bN}\subset\bR$ satisfies $a_1>0$ and $a_k\ge 0, k\ge 2$, then
\begin{align*}
\sum_{n=1}^N\frac{a_n}{(\sum_{k=1}^n a_k)^2}<+\infty.
\end{align*}
\begin{proof}
Note the following
\begin{align*}
\sum_{n=1}^N\frac{a_n}{(\sum_{k=1}^n a_k)^2} &\le \frac1{a_1^2} + \sum_{n=2}^N\frac{a_n}{(\sum_{k=1}^{n-1} a_k)(\sum_{k=1}^n a_k)}\\
&= \frac1{a_1^2} + \sum_{n=2}^N\left(\frac{1}{\sum_{k=1}^{n-1} a_k}-\frac{1}{\sum_{k=1}^n a_k}\right)\\
&= \frac1{a_1^2} + \frac{1}{a_1}-\frac{1}{\sum_{k=1}^N a_k},
\end{align*}
which finishes the proof.
\end{proof}
\end{lemma}

\section{Discussions on derivation of the posterior density}\label{AppendixB}

In this appendix, we will present a formal derivation of the posterior density \eqref{eq:PosteriorTheta}. We make the following standing assumptions
\begin{itemize}
\item [(i)]The random variable $\Theta$ is independent of the Brownian motions $\{w_{k},k\in\bN\}$ and possesses a probability density function $\varrho$;
\item [(ii)] $\sigma(\Theta)\subset\sF_{0}$.
\end{itemize}
We let $\theta$ stand for the dummy variable of $\Theta$. Recall that $U_{N}^{\theta}$ denotes the first $N$ Fourier modes of the solution $u^{\theta}$ to \eqref{eq:mainSPDE} with parameter $\theta$. Let $U_{N}^{\Theta}$ be the $\bR^{N}$-valued process obtained by substituting $\theta$ with $\Theta$ in $U_{N}^{\theta}$. By condition (i), for any $\cB((\bR^{N})^{[0,T]})\otimes\cB(\bR_{+})$-measurable functional $f$ on $(\bR^{N})^{[0,T]}\times\bR_{+}$ with $\bE(f(U_{N}^{\Theta},\Theta))<\infty$, we have that
\begin{align}\label{eq:ExpfUNTheta}
\bE\left(f\left(U_{N}^{\Theta},\Theta\right)\right)=\int_{\bR_{+}}\bE\left(f\left(U^{\theta}_{N},\theta\right)\right)\varrho(\theta)\dif\theta.
\end{align}
Next, we will show that, for any $B\in\cB(\bR_{+})$,
\begin{align}\label{eq:CondThetaUNTheta}
\bP\left(\Theta\in B\mid U^{\Theta}_{N}\right)=\frac{\displaystyle{\int_{B}r_{N}^{\eta}\!\left(U^{\Theta}_{N}\right)\varrho(\eta)\dif\eta}}{\displaystyle{\int_{\bR_{+}}r_{N}^{\eta}\!\left(U^{\Theta}_{N}\right)\varrho(\eta)\dif\eta}},
\end{align}
where,
\begin{align*}
r_{N}^{\eta}\!\left(U_{N}^{\theta}\right):=\exp\left(-\frac{\eta}{\sigma^{2}}\sum_{k=1}^{N}\mu_kq_k^{-2}\int_{0}^{T}\frac{\dif u_{k}^{\theta}(t)}{u_{k}^{\theta}(t)}-\frac{\eta^{2}T}{2\sigma^{2}}\sum_{k=1}^{N}\mu_{k}^{2}q_{k}^{-2}\right),
\end{align*}
for any $\theta,\eta\in\bR_{+}$
For the sake of argument, we assume that $r_{N}^{\eta}(\cdot)$ is $\cB((\bR^{N})^{[0,T]})$-measurable.

Clearly, \eqref{eq:CondThetaUNTheta} follows immediately from the following equality
\begin{align}\label{eq:CondThetaUNTheta2}
\bE\left(\left.\1_{\{\Theta\in B\}}\int_{\bR_{+}}r_{N}^{\eta}\!\left(U^{\Theta}_{N}\right)\varrho(\eta)\dif\eta\,\right|U_{N}^{\Theta}\right)=\int_{B}r_{N}^{\eta}\!\left(U^{\Theta}_{N}\right)\varrho(\eta)\dif\eta.
\end{align}
To prove \eqref{eq:CondThetaUNTheta2}, for any $A\in\cB((\bR^{N})^{[0,T]})$, we first deduce from \eqref{eq:ExpfUNTheta} that
\begin{align*}
\bE\left(\1_{\{U^{\Theta}_{N}\in A\}}\1_{\{\Theta\in B\}}\int_{\bR_{+}}r_{N}^{\eta}\!\left(U^{\Theta}_{N}\right)\varrho(\eta)\dif\eta\right)&=\int_{B}\bE\left(\1_{\{U^{\theta}_{N}\in A\}}\int_{\bR_{+}}r_{N}^{\eta}\big(U^{\theta}_{N}\big)\varrho(\eta)\dif\eta\right)\varrho(\theta)\dif\theta\\
&=\int_{\bR_{+}}\left(\int_{B}\bE\left(\1_{\{U^{\theta}_{N}\in A\}}r_{N}^{\eta}\big(U^{\theta}_{N}\big)\right)\varrho(\theta)\dif\theta\right)\varrho(\eta)\dif\eta.
\end{align*}
In view of \eqref{eq:LogLikelihood}, we have
\begin{align}\label{eq:dPetadPtheta}
\frac{\dif\bP^{\eta}_{N}}{\dif\bP^{\theta}_{N}}\left(U^{\theta}_{N}\right)=\frac{r^{\eta}_{N}\big(U^{\theta}_{N}\big)}{r^{\theta}_{N}\big(U^{\theta}_{N}\big)}.
\end{align}
Hence, by Girsanov theorem,
\begin{align}
\bE\!\left(\!\1_{\{U^{\Theta}_{N}\in A\}}\1_{\{\Theta\in B\}}\!\int_{\bR_{+}}\!r_{N}^{\eta}\!\left(U^{\Theta}_{N}\right)\!\varrho(\eta)\!\dif\eta\!\right)&=\int_{\bR_{+}}\!\!\left(\int_{B}\bE\left(\frac{\dif\bP^{\eta}_{N}}{\dif\bP^{\theta}_{N}}\big(U^{\theta}_{N}\big)\1_{\{U^{\theta}_{N}\in A\}}r^{\theta}_{N}\big(U^{\theta}_{N}\big)\!\right)\!\varrho(\theta)\!\dif\theta\!\right)\!\varrho(\eta)\!\dif\eta\nonumber\\
\label{eq:CondThetaUNThetaLeft} &=\int_{\bR_{+}}\left(\int_{B}\bE\left(\1_{\{U^{\eta}_{N}\in A\}}r^{\theta}_{N}\big(U^{\eta}_{N}\big)\right)\varrho(\theta)\dif\theta\right)\varrho(\eta)\dif\eta.
\end{align}
On the other hand,
\begin{align}
\bE\left(\1_{\{U^{\Theta}_{N}\in A\}}\int_{B}r^{\eta}_{N}\!\left(U^{\Theta}_{N}\right)\varrho(\eta)\dif\eta\right)&=\int_{\bR_{+}}\bE\left(\1_{\{U^{\theta}_{N}\in A\}}\int_{B}r^{\eta}_{N}\big(U^{\theta}_{N}\big)\varrho(\eta)\dif\eta\right)\varrho(\theta)\dif\theta\nonumber\\
\label{eq:CondThetaUNThetaRight} &=\int_{\bR_{+}}\left(\int_{B}\bE\left(\1_{\{U^{\theta}_{N}\in A\}}r^{\eta}_{N}\big(U^{\theta}_{N}\big)\right)\varrho(\eta)\dif\eta\right)\varrho(\theta)\dif\theta.
\end{align}
Combining \eqref{eq:CondThetaUNThetaLeft} and \eqref{eq:CondThetaUNThetaRight} leads to \eqref{eq:CondThetaUNTheta2}.

%\section*{Acknowledgments}
%The authors would like to thank the two anonymous referees, the associate editor and the editor for carefully reading the original manuscript, and their insightful comments and suggestions which improved greatly the final manuscript.
%
%

\bibliographystyle{alpha} % standard LaTeX bibliography format. Preferable to be used
%\bibliography{igor_bib_probability-02-22-2019}

\def\cprime{$'$}

\end{document}